\documentclass{amsart}

\usepackage{amssymb}
\usepackage{amsthm}
\usepackage{amsmath}
\usepackage{amsbsy}
\usepackage[all]{xy}
\usepackage{bm}
\usepackage{hyperref}
\usepackage{tikz}
\usepackage{array}
\usepackage{colortbl,xcolor}
\newtheorem{theorem}{Theorem}[section]
\newtheorem{proposition}[theorem]{Proposition}
\newtheorem{corollary}[theorem]{Corollary}

\newtheorem*{conjecture}{Conjecture}
\newtheorem{lemma}[theorem]{Lemma}

\begin{document}

\title[]{Harborth Constants for Certain Classes of Metacyclic Groups} 

\author[]{Noah Kravitz}
\address[]{Grace Hopper College, Yale University, New Haven, CT 06510, USA}
\email{noah.kravitz@yale.edu}

\begin{abstract} 
The Harborth constant of a finite group $G$  is the smallest integer $k\geq \exp(G)$ such that any subset of $G$ of size $k$ contains $\exp(G)$ distinct elements whose product is $1$.  Generalizing previous work on the Harborth constants of dihedral groups, we compute the Harborth constants for the metacyclic groups of the form $H_{n, m}=\langle x, y \mid x^n=1, y^2=x^m, yx=x^{-1}y \rangle$.  We also solve the ``inverse'' problem of characterizing all smaller subsets that do not contain $\exp(H_{n,m})$ distinct elements whose product is $1$.
\end{abstract}
\maketitle

\section{Introduction and Main Results}

\subsection{Background}


The \textit{exponent} of a finite group $G$ (written $\exp(G)$) is the least common multiple of the orders of the elements of $G$.  The \textit{Harborth constant} of a finite group $G$ (written $\mathfrak{g}(G)$) is defined to be the smallest integer $k\geq \exp(G)$ such that any subset of $G$ of size $k$ contains $\exp(G)$ distinct elements whose product is $1$.  If no such $k\leq |G|$ exists, we say that $\mathfrak{g}(G)=|G|+1$.
\\

The computation of Harborth constants falls under the general category of zero-sum problems.  For a finite additive abelian group $G$, a typical zero-sum problem asks for the smallest positive integer $k$ such that any sequence of $k$ (not necessarily distinct) elements of $G$ contains a subsequence whose terms sum to $0$ while also fulfilling certain other properties.  The most celebrated result in this area is the Erd\H{o}s-Ginzburg-Ziv (EGZ) Theorem \cite{egz}, which says that in any set of $2n-1$ integers, there are $n$ whose sum is divisible by $n$, whereas the same is not always true of a set of $2n-2$ integers.  Various types of zero-sum problems have been extensively studied; for an overview, see the excellent survey articles by Caro \cite{caro} and Gao and Geroldinger \cite{gandg}.  Among the many variants, the Davenport constant of $G$ (written $\mathsf{D}(G)$), the smallest positive integer $k$ such that any sequence of $k$ elements contains a non-empty zero-sum subsequence, has garnered substantial interest (see, e.g., \cite{alon}, \cite{davenport}, \cite{rankthree}).
\\

Perhaps the most popular zero-sum problem concerns the Erd\H{o}s-Ginzburg-Ziv (EGZ) constant (written $\mathsf{s}(G)$), which is defined to be the smallest integer $k \geq \exp(G)$ such that any sequence of $k$ elements contains a zero-sum subsequence of length $\exp(G)$.  (The Harborth constant differs from the EGZ constant in that the former considers only subsets of $G$ without repetition, whereas the latter allows multisubsets of $G$.)  In this language, the EGZ Theorem says that $\mathsf{s}(\mathbb{Z}/n\mathbb{Z})=2n-1$.  In 1973, Harborth \cite{harborth} considered the more general problem of computing $\mathsf{s}((\mathbb{Z}/n\mathbb{Z})^d)$, which can be interpreted in terms of finding sets of lattice points in $d$-dimensional Euclidean space whose centroids are also lattice points.  He gives the bounds $$(n-1)2^d+1 \leq \mathsf{s}((\mathbb{Z}/n\mathbb{Z})^d) \leq (n-1)n^d+1,$$ where the EGZ Theorem shows that this lower bound is sharp for $d=1$.  Reiher \cite{reiher} proved Kemnitz's Conjecture, which states that $\mathsf{s}((\mathbb{Z}/n\mathbb{Z})^2)=4n-3$ also achieves Harborth's lower bound.  Alon and Dubiner \cite{alonlattice}, Elsholtz \cite{elsholtz}, and Fox and Sauermann \cite{fox} have improved the bounds on $\mathsf{s}((\mathbb{Z}/n\mathbb{Z})^d)$.
\\

Much less is known about the Harborth constant than about the EGZ constant despite their similarities.  Marchan, Ordaz, Ramos, and Schmid \cite{abelian}, \cite{inverse} computed the Harborth constants for some classes of abelian groups, in particular,
$$\mathfrak{g}(C_2 \oplus C_{2n})=
\begin{cases}
2n+3, &n \text{ odd}\\
2n+2, &n \text{ even}
\end{cases}$$
for all $n \geq 1$.  They also consider the plus-minus weighted analogue of the Harborth constant, in which one may add either an element or its inverse in the sums of length $\exp(G)$.  For this variant, they compute $\mathfrak{g}_{\pm}(C_2 \oplus C_{2n})=2n+2$ for all $n\geq 3$.  The problem of computing Harborth constants for general finite abelian groups, however, remains wide open.
\\

To date, zero-sum problems have been studied almost exclusively in the context of abelian groups.  The limited existing literature addressing nonabelian groups focuses on dihedral groups and their generalizations.  The earliest papers, due to Zhuang and Gao \cite{zhuang} and Gao and Lu \cite{lu}, dealt with the Davenport constant and the EGZ constant.  Bass \cite{bass} computed the EGZ constants of all dihedral groups of order $2n$ ($\mathsf{s}(D_{2n})=3n$), dicyclic groups of order $4n$ ($\mathsf{s}(D_{4n})=6n$), and general nonabelian groups of order $pq$ ($\mathsf{s}(G_{pq})=pq+p+q-2$).  More recent results in this area are due to Mart\'inez and Ribas \cite{martinez}.
\\

Recently, Balachandran, Mazumdar, and Zhao \cite{dihedral} considered Harborth constants for nonabelian groups.  In particular, they give values for the Harborth constants of dihedral groups ($n \geq 3$):
$$\mathfrak{g}(D_{2n})=\begin{cases}
n+2, &n \text{ even}\\
2n+1, &n \text{ odd}.
\end{cases}$$
Unfortunately, the proof in \cite{dihedral} is incorrect, and rectifying the errors is nontrivial.  In this paper, we build on the techniques developed by Balachandran, Mazumdar, and Zhao in order to prove more general results about metacyclic groups, and the Harborth constants of dihedral groups are a special case of our results.

\subsection{Notation}

The general finite metacyclic group has the presentation $$H_{n, p, m, r}=\langle x, y \mid x^n=1, y^p=x^m, yx=x^{-r}y \rangle$$ with order $|H_{n, p, m, r}|=np$ (where $n, p \geq 2$ and $m, r \geq 0$).  In this paper, we focus on the classes of metacyclic groups where $p=2$ and $r=1$, in which case we write $$H_{n, m}=\langle x, y \mid x^n=1, y^2=x^m, yx=x^{-1}y \rangle.$$
Using the third relation to ``push'' the $y$'s to the right, we express the $2n$ elements of $H_{n, m}$ in the normal form $$H_{n, m}=\{1, x, x^2, \dots, x^{n-1}, y, xy, \dots, x^{n-1}y\}.$$

\textbf{Remarks.}
\begin{enumerate}
\item $H_{2m, m}$ gives the dicyclic groups of order $4m$, and these are the generalized quaternion groups when $m\geq 2$ is a power of $2$.  Moreover, $H_{n, 0}$ gives the dihedral groups of order $2n$.
\item It is not difficult to verify that any finite group with a nontrivial cyclic subgroup of index $2$ is isomorphic to some $H_{n, 2, m, r}$.
\end{enumerate}

For a subset $S$ of a finite group $G$ and an integer $0 \leq t \leq |S|$, we let $\prod_t(S)$ denote the set of all $t-$fold products of distinct elements of $S$.

\subsection{Contributions of this paper}

The following theorem gives the Harborth constants for general $H_{n, m}$.  The case of $n$ and $m$ even, which includes the dihedral and dicyclic groups with order divisible by $4$, is of greatest interest.

\newtheorem*{main}{Theorem \ref{main}}
\begin{main}[Main Theorem]
Let $n\geq 2$ and $m$ be integers.  Then:
$$
\mathfrak{g}(H_{n, m})=
\begin{cases}
2n+1, & n \text{ odd}\\
2n, &n \equiv 0 \pmod{4}, m \text{ odd}\\
2n+1, &n \equiv 2 \pmod{4}, m \text{ odd}\\
n+2, &n \text{ even, } n\neq 2, m \text{ even}\\
5, &n=2, m \text{ even}.
\end{cases}
$$
\end{main}

In Section $2$, we develop a variety of technical machinery.  In Section $3$, we give a full characterization of ``failing'' subsets $S\subseteq H_{n,m}$ with size $\exp(H_{n,m})\leq |S|<\mathfrak{g}(H_{n,m})$.  In Section $4$, we use the results of Sections $3$ and $4$ to prove the Main Theorem.  In Section $5$, we present further observations and topics for future inquiry.

\section{Auxiliary Results}

In this section, we discuss several technical results that will be of use in later sections.  We begin with a straightforward computation.

\begin{proposition}
For integers $n \geq 2$ and $m$, the exponent of $H_{n,m}$ is given by
$$
\exp(H_{n, m})=
\begin{cases}
n, & n \text{ even, } m \text{ even}\\
2n, &\text{otherwise.}
\end{cases}
$$
\label{exponent}
\end{proposition}

\begin{proof}
Note that $|x|=n$ and $|x^a|$ divides $n$ for all $a$.  For any $0 \leq a \leq n-1$, we have $(x^ay)^2=x^ax^{-a}yy=x^m$, so $$|x^ay|=2|x^m|=\frac{2|x|}{\gcd(n, m)}=\frac{2n}{\gcd(n, m)}.$$  If both $n$ and $m$ are even, then $\gcd(n, m)$ is even and $|x^ay|=\frac{n}{\left(\frac{\gcd(n, m)}{2}\right)}$ divides $n$, which lets us conclude that $\exp(H_{n,m})=n$.  If $n$ and $m$ are not both even, then $\gcd(n,m)$ is odd and $\frac{2n}{\gcd(n, m)}$ contains one more factor of $2$ than $n$ does, which implies that $\exp(H_{n,m})=2n$.

\end{proof}

Balachandran, Mazumdar, and Zhao \cite{dihedral} present a result similar to the following lemma, but both their statement and their proof are incorrect.  We simultaneously rectify their errors and extend the lemma to the more general setting of metacyclic groups.  To the extent that only the second and third cases ($n$ even) will be used later in this paper, the first case is solely of independent interest.  Because the proof is long and the casework is somewhat tedious, we defer the proof to the Appendix.
\\

\begin{lemma}
Let $S=\{x^{\alpha_1}y, \dots, x^{\alpha_t}y\}\subset H_{n,m}$ where $0\leq \alpha_1<\dots <\alpha_t<n$.  We get the following bounds, along with equality conditions for sufficiently large $t$.


\begin{itemize}
\item Suppose $n$ is odd.  Then $$|\prod_t(S)|\geq t,$$ with equality (for $t \geq 4$) exactly when $t$ divides $n$ and $$\{\alpha_1, \dots, \alpha_t\}=\{b+\frac{kn}{t} \mid 0\leq k \leq t-1\}$$ for some integer $0\leq b \leq \frac{n}{t}-1$.
\item Suppose both $n$ and $t$ are even.  Then $$|\prod_t(S)|\geq \frac{t}{2},$$ with equality (for $t \geq 2$) exactly when $t$ divides $n$ and $$\{\alpha_1, \dots, \alpha_t\}=\{b+\frac{kn}{t} \mid 0\leq k \leq t-1\}$$ for some integer $0\leq b \leq \frac{n}{t}-1$.
\item Suppose $n$ is even and $t$ is odd.  Then $$|\prod_t(S)|\geq \frac{t+1}{2},$$ with equality (for $t \geq 5$) exactly when $t+1$ divides $n$ and $$\{\alpha_1, \dots, \alpha_t\}=\{b+\frac{kn}{t+1} \mid 0\leq k \leq t\}\setminus \{b+\frac{\ell n}{t+1}\}$$ for some integers $0\leq b \leq \frac{n}{t+1}-1$ and $0\leq \ell \leq t$.
\end{itemize}
\label{big lemma}
\end{lemma}

\textbf{Remarks.}
\begin{enumerate}
\item We may express the two bounds for even $n$ together as $|\prod_t(S)|\geq \lceil \frac{t}{2} \rceil$.
\item A close examination of the proof reveals that when the equality conditions are satisfied, $\prod_t(S)$ has the following very specific forms.  For $n$ odd (and hence also $t$ odd, since $t$ divides $n$), $$\prod_t(S)=\{x^{\alpha+\frac{kn}{t}}y \mid 0\leq k \leq t-1\}$$ where $\alpha=(\alpha_{\frac{t-1}{2}+1}+\alpha_{\frac{t-1}{2}+2}+\dots +\alpha_t)-(\alpha_1+\alpha_2+\dots +\alpha_{\frac{t-1}{2}})+(\frac{t-1}{2})m$.  For $n$ and $t$ both even, $$\prod_t(S)=\{x^{\alpha+\frac{2kn}{t}} \mid 0\leq k \leq \frac{t}{2}-1\}$$ where $\alpha=(\alpha_{\frac{t}{2}+1}+\alpha_{\frac{t}{2}+2}+\dots +\alpha_t)-(\alpha_1+\alpha_2+\dots +\alpha_{\frac{t}{2}})+(\frac{t}{2})m$.  Finally, for $n$ even and $t$ odd, $$\prod_t(S)=\{x^{\alpha+\frac{2kn}{t+1}}y \mid 0\leq k \leq \frac{t+1}{2}-1\}$$ where $\alpha=(\alpha_{\frac{t-1}{2}+1}+\alpha_{\frac{t-1}{2}+2}+\dots +\alpha_t)-(\alpha_1+\alpha_2+\dots +\alpha_{\frac{t-1}{2}})+(\frac{t-1}{2})m$.
\item In the all but the last case, equality occurs when the $\alpha_i$'s form an arithmetic sequence that ``fills'' $\mathbb{Z}/n\mathbb{Z}$ evenly.  In the last case, equality occurs when the $\alpha_i$'s form such a sequence with exactly one element missing.
\item In the first case, the bound cannot be sharp unless either $t=2$, $t=3$, or $t$ divides $n$.  In the second case, the bound cannot be sharp unless $t$ divides $n$.  In the third case, the bound cannot be sharp unless either $t=3$, $t=4$, or $t+1$ divides $n$.  This property will be useful later.
\item Consider the product $(x^{\beta_1}) \cdots (x^{\beta_s})(x^{\alpha_i}y)$.  If any term $x^{\beta_j}$ is moved to the right of $x^{\alpha_i}y$, then it contributes $x^{-\beta_j}$ to the product instead of $x^{\beta_j}$.  With this in mind, define the set of plus-minus weighted $s$-fold products $$\prod^{\pm}_s(\{x^{\beta_1}, \dots, x^{\beta_s}\})=\{x^{\pm\beta_1 \pm \beta_2 \pm \dots \pm \beta_s}\}$$ where all of the signs on the right-hand side are chosen independently.  In particular, we can choose to negate any $\lfloor \frac{s}{2} \rfloor$ of the terms, so (as is clear from the proof in the Appendix) the lower bounds of Lemma \ref{big lemma} also apply to $| \prod^{\pm}_s(\{x^{\beta_1}, \dots, x^{\beta_s}\})|$.  We leave it to the reader to verify that $b=0$, $b \in \{0, \frac{n}{2t}\}$, $b \in \{0, \frac{n}{2(t+1)}\}$ must be added to the equality conditions in the first, second, and third cases, respectively.  
\end{enumerate}

The following lemma generalizes a well-known result that appears in \cite{dihedral} and \cite{abelian}, among other places.

\begin{lemma}
Let $G$ be a group with finite normal subgroup $N$, and suppose $g_1N$ and $g_2N$ are two (possibly identical) cosets of $N$ in $G$.  If $A \subseteq g_1N$ and $B \subseteq g_2N$ satisfy $|A|+|B|\geq |N|+1$, then $A\cdot B=g_1g_2N$, where $A\cdot B=\{ab \mid a\in A, b \in B\}$.
\label{group}
\end{lemma}

\begin{proof}
Because $N$ is normal, we have $g_2N=Ng_2$ and $g_1g_2N=g_1Ng_2$, and we can write $B=\{b_1g_2, \dots, b_{|B|}g_2\}$ where each $b_i$ is in $N$.  It is clear that $A\cdot B \subseteq g_1g_2N$.  For the other inclusion, fix any $h\in N$, and we will show that $g_1hg_2 \in A \cdot B$.  For any $b=b_ig_2\in B$, we see that $g_1hg_2b^{-1}=g_1hb_i^{-1}\in g_1N$.  Since $|\{g_1hg_2b^{-1} \mid b \in B\}|=|B|$, we conclude that $\{g_1hg_2b^{-1} \mid b \in B\} \cap A$ is nonempty.  So there exist $a\in A$ and $b \in B$ such that $g_1hg_2b^{-1}=a$ and $g_1hg_2=ab$.  Hence, $A\cdot B=g_1g_2N$.
\end{proof}

\section{Characterizing Failing Subsets for $n$ and $m$ Even}

When $\mathfrak{g}(G)$ is strictly larger than $\exp(G)$, a natural ``inverse'' to the Harborth problem is the question of characterizing the subsets $S\subseteq G$ of size $\exp(G) \leq |S|<\mathfrak{g}(G)$ that ``fail'' the Harborth condition.  More formally, we say that a subset $S \subseteq G$ \textit{fails} if $\exp(G) \leq |S|<\mathfrak{g}(G)$ and $S$ does not contain $\exp(G)$ distinct elements whose product is $1$, i.e., $1 \notin \prod_{\exp(G)}(S)$.  We say that $S \subseteq G$ of size $\exp(G) \leq |S|<\mathfrak{g}(G)$ \textit{passes} if it does not fail.  (See \cite{inverse} for a discussion of this concept for $G=C_2 \oplus C_{2n}$.)  In this section, we characterize all failing subsets of $H_{n, m}$.
\\

Recall that $\exp(H_{n,m})=2n=|H_{n,m}|$ when $n$ and $m$ are not both even.  According to the Main Theorem, $\mathfrak{g}(H_{n,m})=2n$ when $n \equiv 0 \pmod{4}$ and $m$ is odd, so there are no failing subsets in this case.  In all other cases where $n$ and $m$ are not both even, $\mathfrak{g}(H_{n,m})=2n+1$, which means that $H_{n,m}$ is the only failing subset.  Thus, the problem of characterizing failing subsets is interesting only when $n$ and $m$ are both even.  Here, $\exp(H_{n,m})=n$ and $\mathfrak{g}(H_{n,m})=n+2$ (for $n\neq 2$) indicate that we are interested in failing subsets of sizes $n$ and $n+1$.  (We address failing subsets of $H_{2,m}$ of size $4$ in Remark $2$ following Corollary \ref{nfailcor}.)

\begin{theorem}
Let $n \geq 2$ and $m$ be even integers, and let $S=\{x^{\beta_1}, \dots, x^{\beta_s}, x^{\alpha_1}y, \dots, x^{\alpha_t}y\} \subset H_{n,m}$ satisfy $|S|=s+t=n$.  Then $S$ fails if and only if it has one of the following five forms:

\begin{enumerate}
\item $s$ and $t$ are odd.\\
\item $s$ and $t$ are even, and $(\beta_1+\dots +\beta_s)+(\alpha_1+\dots +\alpha_t)$ is odd.\\
\item $n \equiv 0 \pmod{4}$, and $s=n$ and $t=0$.\\
\item $n \equiv 4 \pmod{8}$ and $m \equiv 2 \pmod{4}$, and $$S\in \{\{1, x^2, \dots, x^{n-2}, y, x^2y, \dots, x^{n-2}y\}, \{1, x^2, \dots, x^{n-2}, xy, x^3y, \dots, x^{n-1}y\}\}.$$
\item $n \equiv 4 \pmod{8}$ and $m \equiv 0 \pmod{4}$, and $$S\in \{\{x, x^3, \dots, x^{n-1}, y, x^2y, \dots, x^{n-2}y\}, \{x, x^3, \dots, x^{n-1}, xy, x^3y, \dots, x^{n-1}y\}\}.$$
\end{enumerate}
\label{n fail}
\end{theorem}

\begin{proof}
We will first show that the each of these five characterizations is sufficient for $S$ to be failing.  We will then show that if $S$ satisfies none of these characterizations, then some re-ordering of the elements of $S$ yields a product of $1$.
\\

Since re-ordering the elements does not change the parity of the powers of $x$ and $y$ in a product, the first charcterization is clearly sufficient for $S$ to be failing.  The same idea shows the sufficiency of the second characterization once one notes that $(\beta_1+\dots +\beta_s)+(\alpha_1-\dots +(-1)^{t+1}\alpha_t)+(\frac{t}{2})m \equiv (\beta_1+\dots +\beta_s)+(\alpha_1+\dots +\alpha_t) \pmod{2}$.
\\

The sufficiency of the third characterization follows from $$(1)(x)(x^2) \cdots (x^{n-1})=x^{\frac{n(n-1)}{2}}=x^{\frac{n}{2}}\neq 1.$$
(Note that $s=n$ and $t=0$ is a special case of the second characterization when $n \equiv 2 \pmod{4}$.)
\\

(In a similar vein, we can use $(x^{a+b}y)(x^a y)=x^{a+b}x^{-a} y y= x^{m+b}$ to compute $$(x^{n-1}y)(x^{n-2}y)\cdots (xy)(y)=(x^{m+1})^{\frac{n}{2}}=x^{(\frac{n}{2})(m+1)},$$ where $(\frac{n}{2})(m+1)$ is odd if $n \equiv 2 \pmod{4}$ and even if $n \equiv 0 \pmod{4}$.  So $s=0$ and $t=n$ is a special case of the second characterization when $n \equiv 2 \pmod{4}$.)
\\

For the fourth and fifth characterizations, consider all cases where $s=t=\frac{n}{2}$ (for $n \equiv 0 \pmod{4}$, of course) and $$\{\beta_1, \dots, \beta_s\}, \{\alpha_1, \dots, \alpha_t\} \in \{\{0, 2, 4, \dots, n-2\}, \{1, 3, \dots, n-1\}\}.$$
First, compute $$(1)(x^2)(x^4)\dots (x^{n-2})=x^{\frac{(n-2)n}{4}}=x^{(\frac{n}{4})n-\frac{n}{2}}=x^{\frac{n}{2}}.$$
 
By Remarks $2$ and $5$ following Lemma \ref{big lemma}, we have
$$\prod_{\frac{n}{2}}^{\pm}(\{1, x^2, \dots, x^{n-2}\})=
\begin{cases}
\{1, x^4, x^8, \dots, x^{n-4}\}, &n \equiv 0 \pmod{8}\\
\{x^2, x^6, x^{10}, \dots, x^{n-2}\}, &n \equiv 4 \pmod{8}.
\end{cases}$$
Also, since $$(x)(x^3)(x^5)\cdots (x^{n-1})=x^{\frac{n^2}{4}}=x^{(\frac{n}{4})n}=1,$$ the same reasoning tells us that 
$$\prod_{\frac{n}{2}}^{\pm}(\{x, x^3, \dots, x^{n-1}\})=\{1, x^4, x^8, \dots, x^{n-4}\}$$ for all $n$.
We can also compute $$(x^{n-2}y)(x^{n-4}y) \dots (y)=(x^{n-1}y)(x^{n-3}y)\dots (xy)=x^{(\frac{n}{4})(m+2)},$$
where $$(\frac{n}{4})(m+2) \equiv
\begin{cases}
2 \pmod{4}, &m \equiv 0 \pmod{4}, n \equiv 4 \pmod{8}\\
0 \pmod{4}, &\text{otherwise}.
\end{cases}$$
Thus, $$\prod_{\frac{n}{2}}(\{y, \dots, x^{n-2}y\})=\prod_{\frac{n}{2}}(\{xy, \dots, x^{n-1}y\})=
\begin{cases}
\{x^2, x^6, \dots, x^{n-2}\}, &m \equiv 0 \pmod{4}, n \equiv 4 \pmod{8}\\
\{1, x^4, \dots, x^{n-4}\}, &\text{otherwise}.
\end{cases}$$
Recall that $$\prod_n(S)=\prod_s^{\pm}(\{x^{\beta_1}, \dots, x^{\beta_s}\})\cdot\prod_t(\{x^{\alpha_1}y, \dots, x^{\alpha_t}y\}),$$
with set ``multiplication'' as described in Lemma \ref{group}.
It is now easy to verify that $$\prod_n(S)=
\begin{cases}
\{x^2, x^6, \dots, x^{n-2}\}, &\text{in the fourth and fifth characterizations}\\
\{1, x^4, \dots, x^{n-4}\}, &\text{otherwise}.
\end{cases}$$
Finally, $1 \notin \{x^2, x^6, \dots, x^{n-2}\}$ establishes the sufficiency of the fourth and fifth characterizations.
\\

Now, suppose that $S$ does not satisfy any of these five characterizations.  From the first two, we get  that $s$ and $t$ are even and that $$\prod_n(S) \subseteq \{1, x^2, x^4, \dots, x^{n-2}\},$$ where the set on the right-hand side has size $\frac{n}{2}$.
\\

First, consider the case where $n \equiv 0 \pmod{4}$, and $s=0$ and $t=n$. By Lemma \ref{big lemma}, $|\prod_n(S)| \geq \frac{n}{2}$, from which we can conclude that in fact $\prod_n(S)= \{1, x^2, x^4, \dots, x^{n-2}\}$ and hence $1 \in \prod_n(S)$.  The second and third characterizations eliminate all other cases where $\{s,t\}=\{0,n\}$.  (See the parentheticals on the previous page.),
\\

We now restrict our attention to $0<s,t<n$.  Recall that re-ordering the elements in a  product does not change the parity of the power of $x$.  From Lemma \ref{big lemma},
$$ |\prod^{\pm}_s(\{x^{\beta_1}, \dots, x^{\beta_s}\})|+|\prod_t(\{x^{\alpha_1}y, \dots, x^{\alpha_t}y\})| \geq \frac{s}{2}+\frac{t}{2}=\frac{n}{2}.$$
If $$ |\prod^{\pm}_s(\{x^{\beta_1}, \dots, x^{\beta_s}\})|+|\prod_t(\{x^{\alpha_1}y, \dots, x^{\alpha_t}y\})| \geq \frac{n}{2}+1,$$ then Lemma \ref{group} (applied to cosets of $\{1, x^2, x^4, \dots, x^{n-2}\}$) tells us that $$\prod_n(S)= \prod^{\pm}_s(\{x^{\beta_1}, \dots, x^{\beta_s}\})\cdot\prod_t(\{x^{\alpha_1}y, \dots, x^{\alpha_t}y\})=\{1, x^2, x^4, \dots, x^{n-2}\}$$ and hence $1 \in \prod_n(S)$.  So we need to worry about only the equality case of Lemma \ref{big lemma} where $$ |\prod^{\pm}_s(\{x^{\beta_1}, \dots, x^{\beta_s}\})|=\frac{s}{2} \quad \text{and} \quad |\prod_t(\{x^{\alpha_1}y, \dots, x^{\alpha_t}y\})|=\frac{t}{2}.$$
From Remark $4$ after Lemma \ref{big lemma}, we know that this can happen only when both $s$ and $t$ divide $n$.  So the only possibility is $s=t=\frac{n}{2}$ (where $n \equiv 0 \pmod{4}$ since $s$ and $t$ are even).
We thus have $$\{\beta_1, \dots, \beta_s\}, \{\alpha_1, \dots, \alpha_t\} \in \{\{0, 2, 4, \dots, n-2\}, \{1, 3, \dots, n-1\}\}.$$
As discussed above, the fourth and fifth characterizations deal with all the ways for $S$ to fail in this case.  So we can conclude that $S$ passes, and the five characterizations are necessary as well as suficient for $S$ to fail.
\end{proof}

This theorem also lets us precisely count the failing subsets of size $n$.

\begin{corollary}
Let $n \geq 2$ and $m$ be even integers.  Then
$$\#\{\text{failing $S\subset H_{n,m}$ of size $n$}\}=
\begin{cases}
\frac{3}{4}\binom{2n}{n}+\frac{3}{4}\binom{n}{\frac{n}{2}}, &n \equiv 2 \pmod{4}\\
\frac{3}{4}\binom{2n}{n}-\frac{3}{4}\binom{n}{\frac{n}{2}}+1, &n \equiv 0 \pmod{8}\\
\frac{3}{4}\binom{2n}{n}-\frac{3}{4}\binom{n}{\frac{n}{2}}+3, &n \equiv 4 \pmod{8}
\end{cases}$$

Moreover, for any fixed $m$, the probability that $S$ fails satisfies $$\lim_{n \to \infty} \mathbb{P}(S \text{ fails})=\frac{3}{4}$$ if the subsets of size $n$ of each $H_{n,m}$ are chosen uniformly at random.
\label{nfailcor}
\label{n corollary}
\end{corollary}

\textbf{Remarks.}
\begin{enumerate}
\item This corollary contrasts with the case of a product of $n$ distinct elements of $ H_{n,m}$ chosen uniformly at random.  For any choice of $h_1, \dots, h_{n-1}$, at most $1$ choice for $h_n$ out of the remaining $(n+1)$ elements makes $h_1h_2 \dots h_n=1$, so we get $\lim_{n \to \infty} \mathbb{P}(h_1h_2\dots h_n \neq 1)=1$.
\item When $n=2$, there are $\frac{3}{4}\binom{4}{2}+\frac{3}{4}\binom{2}{1}=\frac{3}{4}(6)+\frac{3}{4}(2)=6$ failing subsets of size $2$.  In fact, since $\binom{4}{2}=6$, all subsets of $H_{2, m}$ of size $2$, $3$, and $4$ fail, which immediately implies that $\mathfrak{g}(H_{2, m})=|H_{2,m}|+1=4+1=5$.
\end{enumerate}

\begin{proof}
Fix some $n$ and $m$.  We begin by counting the subsets that fit the first characterization.  If $\{\alpha_1, \dots, \alpha_t\}=\{\beta_1, \dots, \beta_s\}$, then $s$ and $t$ are odd when $n \equiv 2 \pmod{4}$ and even when $n \equiv 0 \pmod{4}$.  Either way, there are $\binom{n}{\frac{n}{2}}$ such subsets.  Now consider the $\binom{2n}{n}-\binom{n}{\frac{n}{2}}$ subsets for which $\{\alpha_1, \dots, \alpha_t\}\neq \{\beta_1, \dots, \beta_s\}$.  I claim that exactly half of these subsets have $s$ and $t$ odd, for there is a bijection between the subsets where $s$ and $t$ are odd and the subsets where $s$ and $t$ are even.  Fix any such subset $S$.  Since $\{\alpha_1, \dots, \alpha_t\}\neq \{\beta_1, \dots, \beta_s\}$, there exists some smallest integer $1\leq k \leq n$ such that $k$ is in exactly one of $\{\alpha_1, \dots, \alpha_t\}$ and $\{\beta_1, \dots, \beta_s\}$.  Switching whether $k$ is an $\alpha_i$ or a $\beta_j$ flips the parity of $s$ and $t$.  Since this map is its own inverse, it is a bijection.  So
$$\#\{\text{$S$ fitting the first characterization}\}=
\begin{cases}
\frac{1}{2}\binom{2n}{n}-\frac{1}{2}\binom{n}{\frac{n}{2}}, & n \equiv 0 \pmod{4}\\
\frac{1}{2}\binom{2n}{n}+\frac{1}{2}\binom{n}{\frac{n}{2}}, & n \equiv 2 \pmod{4}.
\end{cases}$$

We now count the subsets that fit the second characterization.  First, consider subsets $S$ with the property that for all integers $1\leq k \leq \frac{n}{2}$, $$|\{2k-1, 2k\} \cap \{\alpha_1, \dots, \alpha_t\}| \neq 1 \quad \text{and} \quad |\{2k-1, 2k\} \cap \{\beta_1, \dots, \beta_s\}| \neq 1.$$
Clearly, all such subsets have $s$ and $t$ even, and there are $\binom{n}{\frac{n}{2}}$ of them (since any choice of $\frac{n}{2}$ odd elements determines all of the even elements).  Since $(2k-1)+(2k)$ is always odd, $(\beta_1+\dots+\beta_s)+(\alpha_1+\dots+\alpha_t) \equiv \frac{n}{2} \pmod{2}$, where $\frac{n}{2}$ is odd when $n \equiv 2 \pmod{4}$ and even when $n \equiv 0 \pmod{4}$.  Now, consider the other subsets with $s$ and $t$ even.  We can show that exactly half of these subsets have $(\beta_1+\dots+\beta_s)+(\alpha_1+\dots+\alpha_t)$ odd, for there is a bijection between the subsets where this expression is odd and the subsets where it is even.  Fix any such subset $S$.  There is a smallest $k$ such that $|\{2k-1, 2k\} \cap \{\alpha_1, \dots, \alpha_t\}|=1$ or, if no such $k$ exists, a smallest $\ell$ such that $|\{2\ell-1, 2\ell\} \cap \{\beta_1, \dots, \beta_s\}|=1$.  Switching whether $2k-1$ or $2k$ is in $\{\alpha_1, \dots, \alpha_t\}$ (or whether $2\ell-1$ or $2\ell$ is in $\{\beta_1, \dots, \beta_s\}$, respectively) flips the parity of $(\beta_1+\dots+\beta_s)+(\alpha_1+\dots+\alpha_t)$ without changing $s$ and $t$.  As before, this map is its own inverse and hence a bijection.  From the previous paragraph, we know that there are $\frac{1}{2}\binom{2n}{n}+\frac{1}{2}\binom{n}{\frac{n}{2}}$ subsets with $s$ and $t$ even if $n \equiv 0 \pmod{4}$ and $\frac{1}{2}\binom{2n}{n}-\frac{1}{2}\binom{n}{\frac{n}{2}}$ such subsets if $n \equiv 2 \pmod{4}$.  Thus,
$$\#\{\text{$S$ fitting the second characterization}\}=
\begin{cases}
\frac{1}{4}\binom{2n}{n}-\frac{1}{4}\binom{n}{\frac{n}{2}}, & n \equiv 0 \pmod{4}\\
\frac{1}{4}\binom{2n}{n}+\frac{1}{4}\binom{n}{\frac{n}{2}}, & n \equiv 2 \pmod{4}.
\end{cases}$$

The third characterization adds $1$ failing subset when $n \equiv 0 \pmod{4}$.  Taken together, the fourth and fifth characterizations add $2$ failing subsets when $n \equiv 4 \pmod{8}$.  With these small adjustments from the third, fourth, and fifth characterizations, summing the values from the first and second characterizations yields the desired result.
\\

For the asymptotic, recall that $\binom{2n}{n} \sim \frac{4^n}{\sqrt{\pi n}}$ for large $n$ (from Stirling's approximation).  When we divide the number of failing subsets by $\binom{2n}{n}$ (the total number of subsets of size $n$), the $\frac{3}{4}\binom{2n}{n}$ term dominates for large $n$.
\end{proof}

We can characterize failing subsets of size $n+1$ by building on this result.

\begin{theorem}
Let $n \geq 2$ and $m$ be even integers, and let $S=\{x^{\beta_1}, \dots, x^{\beta_s}, x^{\alpha_1}y, \dots, x^{\alpha_t}y\} \subseteq H_{n,m}$ satisfy $|S|=s+t=n+1$.  Then $S$ fails if and only if it has one of the following three forms:

\begin{enumerate}
\item $s$ is even and $t$ is odd, and all of the $\alpha_i$'s have the same parity, and $(\beta_1+\dots +\beta_s)+(\alpha_1+\dots+\alpha_{t-1})$ is odd.\\
\item $s$ is odd and $t$ is even, and all of the $\beta_j$'s have the same parity, and $(\beta_1+\dots +\beta_{s-1})+(\alpha_1+\dots+\alpha_t)$ is odd.\\
\item The third, fourth, or fifth characterization from Theorem \ref{n fail} applies to some subset of $S$ of size $n$.
\end{enumerate}  
\label{n+1 fail}
\end{theorem}

\begin{proof}
As in the proof of Theorem \ref{n fail}, we show the sufficiency of the above characterizations and then their necessity.
\\

In the first characterization, if an $x^{\beta_j}$ element is removed from $S$, then the remaining subset of $n$ elements contains an odd number of $y$'s, so no re-ordering of these $n$ elements yields a product of $1$.  Now, suppose that an $(x^{\alpha_i}y)$ element is removed from $S$.  By the second characterization in Theorem \ref{n fail}, no re-ordering of these $n$ elements yields a product of $1$, either.  Thus, $S$ fails, which shows the sufficiency of the first characterization.  The sufficiency of the second characterization follows in the same manner.
\\

For the third characterization, let $S=T\cup \{h\}$ where $T$ is the failing subset of size $n$ as described in Theorem \ref{n fail} and $h=x^{\gamma}y^{\varepsilon}$ ($0 \leq \gamma \leq n-1$, $\varepsilon \in \{0, 1\}$) is any other element of $H_{n,m}$.  Clearly, if $h$ is removed from $S$, then no re-ordering of the remaining $n$ elements yields a product of $1$.  Now suppose that the removed element is some $x^{\gamma '}y^{\varepsilon '}\in T$ so that the remaining subset of $n$ elements is $(T\cup \{x^{\gamma}y^{\varepsilon}\}) \setminus \{x^{\gamma '}y^{\varepsilon '}\}$.  Recall from the proof of Theorem \ref{n fail} that $T$ contains an even number of $x^{\alpha_i}y$ elements and that the sum of the powers of $x$ of all of the elements of $T$ is even.  In order for $(T\cup \{x^{\gamma}y^{\varepsilon}\}) \setminus \{x^{\gamma '}y^{\varepsilon '}\}$ to retain these two properties, we must have $\epsilon=\epsilon '$ and $\gamma \equiv \gamma ' \pmod{2}$.  But by inspection of the possibilities for $T$, this is impossible: $T$ already contains all elements $x^{\gamma ''}y^{\varepsilon ''}$ such that $\epsilon ''=\epsilon '$ and $\gamma '' \equiv \gamma ' \pmod{2}$.  So $(T\cup \{x^{\gamma}y^{\varepsilon}\}) \setminus \{x^{\gamma '}y^{\varepsilon '}\}$ is a failing subset of size $n$ by either the first or second characterization of Theorem \ref{n fail}.  Thus, the third characterization of Theorem \ref{n+1 fail} is sufficient for $S$ to fail.
\\

Now, suppose $S$ does not satisfy any of these three characterizations.  According to whether $s$ and $t$ are even or odd, we want to remove an $x^{\beta_j}$ element or an $x^{\alpha_i}y$ element, respectively, such that for the remaining $n$ elements, the sum of the powers of $x$ is even.  Since $S$ doesn't satisfy the first or second characterization, this is possible.  Furthermore, this subset of size $n$ does not fit the third, fourth, or fifth, characterization of Theorem \ref{n fail}, so its elements can be re-ordered to yield a product of $1$.  So $S$ passes, and the three characterizations of Theorem \ref{n+1 fail} are necessary as well as sufficient conditions for $S$ to fail.
\end{proof}

Even though we cannot easily use Theorem \ref{n+1 fail} to count the failing subsets of size $n+1$ precisely, we can still find asymptotic upper bounds.

\begin{corollary}
Let $n \geq 2$ and $m$ be even integers.  Then there exists an absolute constant $C>0$ such that $$\mathbb{P}(\text{$S$ fails})<C \sqrt{\frac{n}{2^n}}$$ for all $n$ and $m$ if $S$ is chosen uniformly at random from the subsets of size $n+1$ of each $H_{n,m}$.
\end{corollary}

\textbf{Remarks.}
\begin{enumerate}
\item Since $\mathbb{P}(\text{$S$ fails})$ probably approaches $0$ at a faster rate than what this upper bound suggests, we make no effort to optimize constants.
\item This corollary tells us that for large $n$, failing subsets of size $n+1$ become vanishingly rare, in contrast with failing subsets of size $n$.  We can thus say that subsets of size $n+2$ are robust in the sense that if any single element is removed, the remaining subset of size $n+1$ almost always passes.
\end{enumerate}

\begin{proof}
Fix some $n$ and $m$.  We begin by counting the subsets that fit the first characterization.  Since all of the $\alpha_i$'s have the same parity, $t \leq \frac{n}{2}$ and hence $s \geq \frac{n}{2}+1$.  Note that $\alpha_1+\dots +\alpha_{t-1} \equiv (t-1)\alpha_1 \equiv 0 \pmod{2}$ regardless of the parity of the $\alpha_i$'s since $t$ is odd, so $\beta_1+\dots+\beta_s$ must be odd.  For each $s$, there are $\sim \frac{1}{2}\binom{n}{s}$ such choices for $\{\beta_1, \dots, \beta_s\}$ and then $2\binom{\frac{n}{2}}{n-s+1}$ choices for $\{\alpha_1, \dots, \alpha_t\}$, where $\binom{n}{s} <\binom{n}{\frac{n}{2}}$ and $\binom{\frac{n}{2}}{n-s+1}\leq \binom{\frac{n}{2}}{\lceil \frac{n}{4} \rceil}$.  Since there are $\lceil \frac{n}{4} \rceil$ possible values of $s$, we get, in total, $$\#\{\text{$S$ fitting the first characterization}\} \lesssim \left\lceil \frac{n}{4} \right\rceil \binom{n}{\frac{n}{2}} \binom{\frac{n}{2}}{\lceil \frac{n}{4} \rceil}$$
(where $f(n) \lesssim g(n)$ means that there exists some constant $c$ such that $f(n)<cg(n)$ for all sufficiently large $n$.)  It is clear that there is an equal number of sets satisfying the second condition.  The third characterization contributes at most $3n$ more failing subsets.
\\

Finally, we use Stirling's approximation to bound the asymptotic probability:
\begin{align*}
\mathbb{P}(\text{$S$ fails}) &\lesssim \frac{(2)\lceil \frac{n}{4} \rceil \binom{n}{\frac{n}{2}} \binom{\frac{n}{2}}{\lceil \frac{n}{4} \rceil}+3n}{\binom{2n}{n+1}}\\
 &\lesssim \frac{(\frac{n}{2})(\frac{4^{\frac{n}{2}}}{\sqrt{\frac{\pi n}{2}}})(\frac{4^{\frac{n}{4}}}{\sqrt{\frac{\pi n}{4}}})+3n}{\frac{4^n}{\sqrt{\pi n}}}\\
 &\lesssim \frac{(\frac{n}{2}) \sqrt{8}}{4^{\frac{n}{4}} \sqrt{\pi n}}\\
 &\lesssim \left(\sqrt{\frac{2}{\pi}}\right) \sqrt{\frac{n}{2^n}}
\end{align*}
for large values of $n$.  We can choose $C$ large enough to accomodate both large and small values of $n$.
\end{proof}

\section{Proof of the Main Theorem}

In this section, we prove all parts of the Main Theorem.  Only the fourth part requires substantial work.

\begin{theorem}[Main Theorem]
Let $n\geq 2$ and $m$ be integers.  Then:
$$
\mathfrak{g}(H_{n, m})=
\begin{cases}
2n+1, & n \text{ odd}\\
2n, &n \equiv 0 \pmod{4}, m \text{ odd}\\
2n+1, &n \equiv 2 \pmod{4}, m \text{ odd}\\
n+2, &n \text{ even, } n\neq 2, m \text{ even}\\
5, &n=2, m \text{ even}.
\end{cases}
$$
\label{main}
\end{theorem}

\begin{proof}

We prove the Main Theorem part-by-part.  Only the case of $n\geq 4$ and $m$ even requires substantial work.
\\

\textbf{Case $1$.} $n$ odd.
\\

Recall from Proposition \ref{exponent} that $\exp(H_{n,m})=2n$.  Thus, $\mathfrak{g}(H_{n,m})\in \{2n, 2n+1\}$.  The only subset of $H_{n,m}$ of size $2n$ is $H_{n,m}$ itself.  This subset contains an odd number of $y$'s, so $1 \notin \prod_{2n}(H_{n,m})$, which establishes $\mathfrak{g}(H_{n,m})\geq 2n+1$ and hence $\mathfrak{g}(H_{n,m})=2n+1$.
\\

\textbf{Case $2$.} $n \equiv 0 \pmod{4}$, $m$ odd.
\\

As above, $\exp(H_{n,m})=2n$.  We know that $$(1)(x)\cdots (x^{n-1})(y)(xy)\cdots (x^{n-2}y)(x^{n-1}y)=(x^{\frac{n}{2}})(x^{m-1})^{\frac{n}{2}}=x^{\frac{n}{2}(m)}$$ is an even power of $x$ because $\frac{n}{2}$ is even.  Let $0\leq e \leq n-2$ so that $e \equiv \frac{n}{2}(m) \pmod{n}$.  If we move the $x^{\frac{e}{2}}$ term to the right of the $y$ term, then the resulting product is $x^{\frac{n}{2}(m)-2(\frac{e}{2})}=1$, which establishes $\mathfrak{g}(H_{n,m})=2n$.
\\

\textbf{Case $3$.} $n \equiv 2 \pmod{4}$, $m$ odd.
\\

We still have $\exp(H_{n,m})=2n$.  Since both $\frac{n}{2}$ and $m$ are odd, $$(1)(x)(x^2)\cdots (x^{n-1})(y)(xy)\cdots (x^{n-2}y)(x^{n-1}y)=x^{\frac{n}{2}(m)}$$ is an odd power of $x$.  So $1 \notin \prod_{2n}(H_{n,m})$, which establishes $\mathfrak{g}(H_{n,m})=2n+1$.
\\

\textbf{Case $4$.} $n$ even, $m$ even.
\\

Recall that $\exp(H_{n,m})=n$.  As described in Theorem \ref{n+1 fail}, $H_{n,m}$ contains failing subsets of size $n+1$, so we immediately have $\mathfrak{g}(H_{n,m})\geq n+2$.  Recall from Remark $2$ after Corollary \ref{n corollary} that $\mathfrak{g}(H_{2,m})=5$.  The rest of this proof is devoted to showing that for $n\geq 4$, any subset of $H_{n,m}$ of size $n+2$ contains $n$ distinct elements whose product is $1$.  This will imply that $\mathfrak{g}(H_{n,m})\leq n+2$ and hence $\mathfrak{g}(H_{n,m})=n+2$.  Let $S=\{x^{\beta_1}, \dots, x^{\beta_s}, x^{\alpha_1}y, \dots, x^{\alpha_t}y\} \subset H_{n,m}$ with $|S|=s+t=n+2$.
\\

First, consider $n=4$.  If $s=4$ and $t=2$, then $S=\{1, x, x^2, x^3, x^{\alpha_1}y, x^{\alpha_2}y\}$.  Since $(x)(x^3)=1$, $(1)(x)=x$, $(1)(x^2)=x^2$, and $(1)(x^3)=x^3$, we can always find $0 \leq \beta_1<\beta_2 \leq 3$ such that $(x^{\beta_1})(x^{\beta_2})(x^{\alpha_1}y)(x^{\alpha_2}y)=1$.  If $s=2$ and $t=4$, then $(x^2y)(y)(x^3y)(xy)=x^{2(2+m)}=1$ works.  If $s=t=3$, then, without loss of generality, let $\beta_2$ and $\beta_3$ ($\beta_2<\beta_3$) be of the same parity, and let $\alpha_3$ be of the parity that makes $(x^{\beta_1})(x^{\beta_2})(x^{\alpha_1}y)(x^{\alpha_2}y)=x^{\gamma}$ an even power of $x$.  If $\gamma \equiv 0 \pmod{4}$, then we are done, and if $\gamma \equiv 2 \pmod{4}$, then $(x^{\beta_1})(x^{\beta_3})(x^{\alpha_1}y)(x^{\alpha_2}y)=1$ works.  So we can conclude that $\mathfrak{g}(H_{4, m})=6$, as desired.
\\

Henceforth, consider $n \geq 6$.  I claim that $S$ contains $n$ elements whose product is in $\{1, x^2, x^4, \dots, x^{n-2}\}$.  We know that $\max\{s, t\} \geq \frac{n+2}{2}=\frac{n}{2}+1$, so the larger of $\{\beta_1, \dots, \beta_s\}$ and $\{\alpha_1, \dots, \alpha_t\}$ necessarily contains both even and odd elements.  If $s$ and $t$ are odd, then we can remove an $x^{\beta_j}$ element and an $x^{\alpha_i}y$ element so that the product of the remaining $n$ elements is in $\{1, x^2, x^4, \dots, x^{n-2}\}$.  If $s$ and $t$ are even, then, according to whether $s\geq t$ or $s<t$, we can remove either two $x^{\beta_j}$ elements or two $x^{\alpha_i}y$ elements so that the product of the remaining $n$ elements is in $\{1, x^2, x^4, \dots, x^{n-2}\}$.  This establishes the claim.
\\

Without loss of generality, suppose we removed the elements with the largest $\alpha_i$ and/or $\beta_j$ subscript labels so that the remaining $n$ elements are given by $T=\{x^{\beta_1}, \dots, x^{\beta_u}, x^{\alpha_1}y, \dots, x^{\alpha_v}y\}$, where $2 \leq u, v, \leq n-2$ by construction.  Following the reasoning from the proof of Theorem \ref{n fail}, we see that necessarily $$\prod_n(T)=\{1, x^2, x^4, \dots, x^{n-2}\}$$ (in which case we are done) unless
$$\{\beta_1, \dots, \beta_u\}, \{\alpha_1, \dots, \alpha_v\} \in \{\{0, 2, \dots, n-2\}, \{1, 3, \dots, n-1\}\}.$$

Suppose this is the case, and distinguish possibilities based on the parity of $s$ and $t$.  If $s$ and $t$ are odd, then removing $x^{\beta_1}$ and $x^{\alpha_1}y$ instead of $x^{\beta_s}$ and $x^{\alpha_t}y$ produces a subset $T'$ consisting of $n$ elements whose product is still in $\{1, x^2, x^4, \dots, x^{n-2}\}$ because we have not changed the parity of the total powers of $x$.  But now $$\{\beta_2, \dots, \beta_s\}, \{\alpha_2, \dots, \alpha_t\} \notin \{\{0, 2, \dots, n-2\}, \{1, 3, \dots, n-1\}\},$$
which implies that $$\prod_n(T')=\{1, x^2, x^4, \dots, x^{n-2}\}.$$
If $s$ and $t$ are even, then we have two subcases.  If $s \geq t$, then removing $x^{\beta_1}$ and $x^{\beta_2}$ instead of $x^{\beta_{s-1}}$ and $x^{\beta_s}$ produces a subset $T'$ such that $\prod_n(T')=\{1, x^2, x^4, \dots, x^{n-2}\}.$  (Here, we used $n \geq 6$ to ensure that $\{\beta_3, \dots, \beta_s\}$ is not a different forbidden arithmetic sequence.)  Similarly, if $s<t$, then removing $x^{\alpha_1}y$ and $x^{\alpha_2}y$ instead of $x^{\alpha_{t-1}}y$ and $x^{\alpha_t}y$ produces such a subset $T'$.  This completes the casework and lets us conclude that $\mathfrak{g}(H_{n,m})=n+2$.
\end{proof}

\section{Additional Results and Topics for Future Inquiry}

\subsection{An extension of the Main Theorem}

A close examination of the proof of the Main Theorem for $n$ and $m$ even reveals the following stronger statement.

\begin{corollary}
Let $n\geq 4$ and $m$ be even integers.  Then for any subset $S\subset H_{n,m}$ satisfying $|S|=n+2$, we have $$\prod_n(S)=H_{n,m}.$$
\end{corollary}

\begin{proof}
Consider the various elements $h \in H_{n,m}$.  The proof of the Main Theorem makes it clear that $h\in \prod_n(S)$ when $h=x^a$ for $a$ even.  When $h=x^a$ for $a$ odd, we can adapt the proof so that in our subsets of size $n$, the sum of the powers of $x$ is odd, after which we can apply the same reasoning about avoiding forbidden arithmetic progressions.  When $h=x^ay$, the proof idea is even easier.  We can choose to remove two elements of $S$ so that in the (re-indexed) remaining set $T=\{x^{\beta_1}, \dots, x^{\beta_u}, x^{\alpha_1}y, \dots, x^{\alpha_v}y\}$, $u$ and $v$ are odd and $(\beta_1+\dots+\beta_u)+(\alpha_1+\dots+\alpha_v)$ is the same parity as $a$.  The third case of Lemma \ref{big lemma} gives
$$ |\prod^{\pm}_u(\{x^{\beta_1}, \dots, x^{\beta_u}\})|+|\prod_v(\{x^{\alpha_1}y, \dots, x^{\alpha_v}y\})| \geq \frac{u+1}{2}+\frac{v+1}{2}=\frac{n}{2}+1,$$
at which point an application of Lemma \ref{group} yields the desired result.
\end{proof}

Investigating this ``stronger'' Harborth condition in other groups would be interesting.

\subsection{Connection to the plus-minus weighted analogue}
The ability of the powers of $x$ to contribute either positively or negatively to products is reminiscent of the plus-minus weighting discussed in \cite{abelian}.  It is easy to see that $\mathfrak{g}_{\pm}(H_{n, m})=n+2$ when $n \geq 6$ and $m$ are both even.  (The plus-minus weighted Harborth constant is always less than or equal to the ordinary Harborth constant.  For $n \equiv 2 \pmod{4}$, $\{1, x, \dots, x^{n-1}, y\}$ is a failing subset of size $n+1$, and for $n \equiv 0 \pmod{4}$, $\{1, x, \dots x^{n-3}, y, x^2y, x^4y\}$ is such a subset.)  So, in a sense, the commutator relation $yx=x^{-1}y$ builds in enough ``flexibility'' that the Harborth constant is stable under the introduction of plus-minus weightings.  Further analogies to the discussion in \cite{abelian} could also be of future interest.

\subsection{Erd\H{o}s-Ginzburg-Ziv constants}
Recall from the Introduction that Bass \cite{bass} computed the EGZ constants of all dihedral and dicyclic groups: in both cases, $\mathsf{s}(G)=\frac{3}{2}|G|$.  In the same paper, he suggests the (still open) problem of computing the EGZ constants of other semidirect products of cyclic groups.  In light of the present results on Harborth constants, we think that generalization to the metacyclic groups discussed in this paper might be more fruitful.  In particular, we present the following conjecture.

\begin{conjecture}
Let $n\geq 4$ and $m$ be even integers.  Then $\mathsf{s}(H_{n,m})=\frac{3}{2}|H_{n,m}|=3n.$
\end{conjecture}

\subsection{Other nonabelian groups}

A natural topic for further inquiry is computing the Harborth constants for metacyclic groups $H_{n, p, m, r}=\langle x, y \mid x^n=1, y^p=x^m, yx=x^{-r}y \rangle$ for other values of $p$ and $r$.  We believe that even if exact results are difficult to compute, it should be possible to develop good bounds in some cases.  The techniques of Lemma \ref{big lemma} seem especially promising for the case $r^2 \equiv 1 \pmod{n}$.  Other classes of supersolvable nonabelian groups, such as the generalized dihedral groups, are also good candidates for computing Harborth constants because they admit simple normal forms.

\section*{Appendix: Proof of Lemma \ref{big lemma}}

In this appendix, we prove all of the parts of Lemma \ref{big lemma}.  The bulk of the proof is devoted to showing the necessity of the equality condition for $n$ even and $t$ odd.  Even though this part of Lemma \ref{big lemma} is not used elsewhere in this paper, we include it both because it is of independent interest and because it corrects an error in \cite{dihedral}.

\begin{proof}
We distinguish four cases based on the parity of $n$ and $t$.\\

\textbf{Case $1$.} $n$ odd, $t$ even.\\

First, consider $t=0$.  $\prod_0(S)=\{1\}$ and $|\prod_0(S)|=1\geq 0$.  Now, consider $t=2$.  $\prod_2(S)=\{x^{\alpha_2-\alpha_1+m}, x^{\alpha_1-\alpha_2+m}\}$ where $\alpha_2-\alpha_1$ and $\alpha_1-\alpha_2$ are distinct modulo $n$ since $n$ is odd.  So $|\prod_2(S)|=2$ for all choices of $0\leq \alpha_1<\alpha_2<n$.
\\

Henceforth, consider $t\geq 4$.  Write $$(x^{\alpha_{\frac{t}{2}+1}}y)(x^{\alpha_1}y)(x^{\alpha_{\frac{t}{2}+2}}y)(x^{\alpha_2}y)\dots (x^{\alpha_t}y)(x^{\alpha_{\frac{t}{2}}}y)=x^{\alpha}\in \prod_t(S)$$ where $\alpha=(\alpha_{\frac{t}{2}+1}+\alpha_{\frac{t}{2}+2}+\dots +\alpha_t)-(\alpha_1+\alpha_2+\dots +\alpha_{\frac{t}{2}})+(\frac{t}{2})m$.
For $1 \leq i, j \leq \frac{t}{2}$, swapping the terms $(x^{\alpha_i}y)$ and $(x^{\alpha_{\frac{t}{2}+j}}y)$ gives $$x^{\alpha -2(\alpha_{\frac{t}{2}+j}-\alpha_i)}\in \prod_t (S).$$
Note that $$0<\alpha_{\frac{t}{2}+1}-\alpha_{\frac{t}{2}}<\alpha_{\frac{t}{2}+1}-\alpha_{\frac{t}{2}-1}<\dots <\alpha_{\frac{t}{2}+1}-\alpha_1<\alpha_{\frac{t}{2}+2}-\alpha_1<\dots <\alpha_t-\alpha_1<n.$$
Since $n$ is odd, these $t$ differences (including $0$) remain distinct modulo $n$ when they are doubled, so $|\prod_t(S)|\geq t$, as desired.
\\

For the sake of contradiction, suppose there is equality.  Then the $t$ elements of $\prod_t(S)$ are exactly $$\prod_t(S)=\{x^{\alpha-2(0)}, x^{\alpha-2(\alpha_{\frac{t}{2}+1}-\alpha_{\frac{t}{2}})}, \dots, x^{\alpha-2(\alpha_{\frac{t}{2}+1}-\alpha_1)}, x^{\alpha-2(\alpha_{\frac{t}{2}+2}-\alpha_1)}, \dots, x^{\alpha-2(\alpha_t-\alpha_1)}\}.$$
In particular, $x^{\alpha-2(\alpha_{\frac{t}{2}+2}-\alpha_2)}$ is on this list.  (Here, we used $t\geq 4$.)  From $$\alpha_{\frac{t}{2}+1}-\alpha_2<\alpha_{\frac{t}{2}+2}-\alpha_2<\alpha_{\frac{t}{2}+2}-\alpha_1,$$ we can conclude that $$\alpha_{\frac{t}{2}+2}-\alpha_2=\alpha_{\frac{t}{2}+1}-\alpha_1 \quad \text{and} \quad \alpha_{\frac{t}{2}+2}-\alpha_{\frac{t}{2}+1}=\alpha_2-\alpha_1.$$
Next, $\alpha_{\frac{t}{2}+1}-\alpha_3<\alpha_{\frac{t}{2}+2}-\alpha_3<\alpha_{\frac{t}{2}+2}-\alpha_2=\alpha_{\frac{t}{2}+1}-\alpha_1$ implies $\alpha_{\frac{t}{2}+2}-\alpha_3=\alpha_{\frac{t}{2}+1}-\alpha_2$ and $\alpha_{\frac{t}{2}+2}-\alpha_{\frac{t}{2}+1}=\alpha_3-\alpha_2$.  Continuing this process leads to $$\alpha_{\frac{t}{2}+2}-\alpha_{\frac{t}{2}+1}=\alpha_2-\alpha_1=\alpha_3-\alpha_2=\dots =\alpha_{\frac{t}{2}}-\alpha_{\frac{t}{2}-1}.$$
Similarly, $\alpha_{\frac{t}{2}+1}-\alpha_1=\alpha_{\frac{t}{2}+2}-\alpha_2<\alpha_{\frac{t}{2}+3}-\alpha_2<\alpha_{\frac{t}{2}+3}-\alpha_1$ implies $\alpha_{\frac{t}{2}+3}-\alpha_2=\alpha_{\frac{t}{2}+2}-\alpha_1$ and $\alpha_{\frac{t}{2}+3}-\alpha_{\frac{t}{2}+2}=\alpha_2-\alpha_1$.  As above, continuing this process leads to $$\alpha_2-\alpha_1=\alpha_{\frac{t}{2}+3}-\alpha_{\frac{t}{2}+2}=\alpha_{\frac{t}{2}+4}-\alpha_{\frac{t}{2}+3}=\dots =\alpha_t-\alpha_{t-1}.$$
Furthermore, we may consider swapping two disjoint pairs of elements at once.  In particular, $\alpha_t-\alpha_1<(\alpha_t-\alpha_1)+(\alpha_{\frac{t}{2}+1}-\alpha_{\frac{t}{2}})<n+(\alpha_{\frac{t}{2}+1}-\alpha_{\frac{t}{2}})$ implies $(\alpha_t-\alpha_1)+(\alpha_{\frac{t}{2}+1}-\alpha_{\frac{t}{2}})=n$.
\\

Now, consider $t=4$.  The previous equation becomes $(\alpha_4-\alpha_1)+(\alpha_3-\alpha_2)=n$.  Using $\alpha_4-\alpha_2=\alpha_3-\alpha_1$ (from above), we get $2(\alpha_3-\alpha_1)=n$, but this is impossible since $n$ is odd.
\\

Henceforth, consider $t \geq 6$.  As in the previous paragraph, $n=(\alpha_t-\alpha_1)+(\alpha_{\frac{t}{2}+1}-\alpha_{\frac{t}{2}})<(\alpha_t-\alpha_1)+(\alpha_{\frac{t}{2}+1}-\alpha_{\frac{t}{2}-1})<n+(\alpha_{\frac{t}{2}+1}-\alpha_{\frac{t}{2}-1})$ implies $(\alpha_t-\alpha_1)+(\alpha_{\frac{t}{2}+1}-\alpha_{\frac{t}{2}-1})=n+(\alpha_{\frac{t}{2}+1}-\alpha_{\frac{t}{2}})$.  (All of these elements are distinct because $t \geq 6$.)  Substituting for $n$ gives
$$(\alpha_t-\alpha_1)+(\alpha_{\frac{t}{2}+1}-\alpha_{\frac{t}{2}-1})=((\alpha_t-\alpha_1)+(\alpha_{\frac{t}{2}+1}-\alpha_{\frac{t}{2}}))+(\alpha_{\frac{t}{2}+1}-\alpha_{\frac{t}{2}}),$$
which after cancellations becomes
$$\alpha_{\frac{t}{2}}-\alpha_{\frac{t}{2}-1}=\alpha_{\frac{t}{2}+1}-\alpha_{\frac{t}{2}}.$$
Putting everything together, we get 
$$\alpha_2-\alpha_1=\alpha_3-\alpha_2=\dots=\alpha_t-\alpha_{t-1}$$ and hence
$$\alpha_1=b, \quad \alpha_2=b+d, \quad \alpha_3=b+2d, \quad \dots, \quad \alpha_t=b+(t-1)d$$ for some integers $b$ and $d$.  We can now read off the $t$ differences as 
$$0<d<2d<\dots<(t-1)d<n.$$
Evaluating $$(\alpha_t-\alpha_1)+(\alpha_{\frac{t}{2}+1}-\alpha_{\frac{t}{2}})=(t-1)d+d=td=n$$
lets us conclude that $td=n$ and $d=\frac{n}{t}$, but this is impossible since $t$ (which is even) does not divide $n$ (which is odd).  So there cannot be equality for even $t \geq 4$.
\\

\textbf{Case $2$.} $n$ odd, $t$ odd.\\

First, note that $|\prod_1(S)|=1$ is trivially true for $t=1$.  Now, consider $t=3$.  Let $\alpha=\alpha_1+\alpha_2+\alpha_3+m$ so that $\prod_3(S)=\{x^{\alpha-2\alpha_1}y, x^{\alpha-2\alpha_2}y, x^{\alpha-2\alpha_3}y\}$, where the exponents are distinct modulo $n$ because $n$ is odd.  So $|\prod_3(S)|=3$ for all choices of $0\leq \alpha_1<\alpha_2<\alpha_3<n$.
\\

Henceforth, consider $t \geq 5$.  Write $$(x^{\alpha_{\frac{t-1}{2}+1}}y)(x^{\alpha_1}y)(x^{\alpha_{\frac{t-1}{2}+2}}y)(x^{\alpha_2}y)\cdots (x^{\alpha_{t-1}}y)(x^{\alpha_{\frac{t-1}{2}}}y)(x^{\alpha_t}y)=x^{\alpha}y\in \prod_t(S)$$ where $\alpha=(\alpha_{\frac{t-1}{2}+1}+\alpha_{\frac{t-1}{2}+2}+\dots +\alpha_t)-(\alpha_1+\alpha_2+\dots +\alpha_{\frac{t-1}{2}})+(\frac{t-1}{2})m$.   Following Case $1$, we have $$x^{\alpha-2(\alpha_{\frac{t-1}{2}+j}-\alpha_i)}y \in \prod_t(S)$$ for $1 \leq i \leq \frac{t-1}{2}$, $1 \leq j \leq \frac{t+1}{2}$.
Also as above, we have $$0<\alpha_{\frac{t-1}{2}+1}-\alpha_{\frac{t-1}{2}}<\dots <\alpha_{\frac{t-1}{2}+1}-\alpha_1<\alpha_{\frac{t-1}{2}+2}-\alpha_1<\dots <\alpha_t-\alpha_1<n,$$ where these $t$ differences (including $0$) remain distinct modulo $n$ when they are doubled, so $|\prod_t(S)|\geq t$.
\\

Suppose there is equality.  The reasoning used in Case $1$ shows that
$$\alpha_{\frac{t-1}{2}+2}-\alpha_{\frac{t-1}{2}+1}=\alpha_2-\alpha_1=\alpha_3-\alpha_2=\dots =\alpha_{\frac{t-1}{2}}-\alpha_{\frac{t-1}{2}-1}$$ and
$$\alpha_2-\alpha_1=\alpha_{\frac{t-1}{2}+3}-\alpha_{\frac{t-1}{2}+2}=\alpha_{\frac{t-1}{2}+4}-\alpha_{\frac{t-1}{2}+3}=\dots =\alpha_t-\alpha_{t-1}.$$
Also as in Case $1$, we can swap disjoint pairs of elements.  In particular, $(\alpha_t-\alpha_1)+(\alpha_{\frac{t-1}{2}+1}-\alpha_{\frac{t-1}{2}})=n$ and $(\alpha_t-\alpha_1)+(\alpha_{\frac{t-1}{2}+2}-\alpha_{\frac{t-1}{2}})=n+(\alpha_{\frac{t-1}{2}+1}-\alpha_{\frac{t-1}{2}})$.  Substituting for $n$ gives
$$(\alpha_t-\alpha_1)+(\alpha_{\frac{t-1}{2}+2}-\alpha_{\frac{t-1}{2}})=((\alpha_t-\alpha_1)+(\alpha_{\frac{t-1}{2}+1}-\alpha_{\frac{t-1}{2}}))+(\alpha_{\frac{t-1}{2}+1}-\alpha_{\frac{t-1}{2}}),$$ which after cancellations becomes
$$\alpha_{\frac{t-1}{2}+2}-\alpha_{\frac{t-1}{2}+1}=\alpha_{\frac{t-1}{2}+1}-\alpha_{\frac{t-1}{2}}.$$
Putting everything together gives $$\alpha_2-\alpha_1=\alpha_3-\alpha_2=\dots=\alpha_t-\alpha_{t-1}$$
and (following Case $1$)
$$\alpha_1=b, \quad \alpha_2=b+\frac{n}{t}, \quad \alpha_3=b+\frac{2n}{t}, \quad \dots, \quad \alpha_t=\frac{(t-1)n}{t}$$ where $0 \leq b \leq \frac{n}{t}-1$ and (obviously) $t$ divides $n$, as desired.
\\

To see that this necessary condition is also sufficient, let the $\alpha_i$'s be given as above.  Then any element of $\prod_t(S)$ is of the form $x^{\alpha+2\ell(\frac{n}{t})}y$ for some integer $\ell$.  Since $n$ is odd, $\frac{2n}{t}$ has additive order $t$ modulo $n$, an the exponent can assume at most $t$ distinct values modulo $n$.  Thus, $|\prod_t(S)|\leq t$ and in fact $|\prod_t(S)|=t$.
\\

\textbf{Case $3$.} $n$ even, $t$ even.\\

First, consider $t=0$.  $\prod_0(S)=\{1\}$ and $|\prod_0(S)|=1\geq \frac{0}{2}$.  Now, consider $t=2$.  $\prod_2(S)=\{x^{\alpha_2-\alpha_1+m}, x^{\alpha_1-\alpha_2+m}\}$, where these two elements coincide exactly when $\alpha_2-\alpha_1=\frac{n}{2}$.  So we always have $|\prod_2(S)|\geq 1= \frac{2}{2}$, with equality exactly when $\{\alpha_1, \alpha_2\}=\{b, b+\frac{n}{2}\}$ for some $0\leq b \leq \frac{n}{2}-1$.
\\

Henceforth, consider $t \geq 4$.  Define $\alpha$ as in Case $1$ so that $x^{\alpha} \in \prod_t (S)$ and $x^{\alpha -2(\alpha_{\frac{t}{2}+j}-\alpha_i)}\in \prod_t (S)$ for all $1 \leq i, j \leq \frac{t}{2}$.   Recall the list of $t$ differences from Case $1$.  Because $n$ is now even, we are guaranteed only $\frac{t}{2}$ distinct values modulo $n$ when we double all the differences.  Hence, $|\prod_t(S)|\geq \frac{t}{2}$.
\\

Suppose there is equality.  We must have $|\{0\}\cup \{\alpha_{\frac{t}{2}+j}-\alpha_i \mid 1 \leq i, j \leq \frac{t}{2}\}|=t$ since if this set were any larger, doubling all the values would yield more than $\frac{t}{2}$ distinct sums modulo $n$.  So this set equals the set of the $t$ listed differences.
\\

Now, consider $t=4$.  From the argument in Case $1$, we have $\alpha_2-\alpha_1=\alpha_4-\alpha_3$ and $2(\alpha_3-\alpha_1)=n$, which implies $\alpha_3-\alpha_1=\frac{n}{2}$.  Furthermore, because double-counting requires the $4$ listed differences come in pairs separated by $\frac{n}{2}$, we also have $\alpha_4-\alpha_1=\frac{n}{2}+(\alpha_3-\alpha_2)$, which implies $2(\alpha_2-\alpha_1)=\frac{n}{2}$ and $\alpha_2-\alpha_1=\frac{n}{4}$.  Now, we can read off
$$\alpha_1=b, \quad \alpha_2=b+\frac{n}{4}, \quad \alpha_3=b+\frac{2n}{4}, \quad \alpha_4=b+\frac{3n}{4},$$ where $0 \leq b \leq \frac{n}{4}-1$ and $4$ divides $n$, as desired.
\\

For $t \geq 6$, the argument of Case $1$ gives $$\{\alpha_1, \dots, \alpha_t\}=\{b+\frac{kn}{t} \mid 0\leq k \leq t-1\}$$ where $0\leq b \leq \frac{n}{t}-1$ and $t$ divides $n$.  To see the sufficiency of this condition (for all $n \geq 4$), let the $\alpha_i$'s be given as above.  Analogously to Case $2$, any element of $\prod_t(S)$ is of the form $x^{\alpha+2\ell(\frac{n}{t})}$.  Since $t$ is even, $\frac{2n}{t}$ divides $n$ and has order $\frac{t}{2}$ modulo $n$.  So the exponent assumes at most $\frac{t}{2}$ distinct values modulo $n$, and $|\prod_t(S)| \leq \frac{t}{2}$, as desired.
\\

\textbf{Case $4$.} $n$ even, $t$ odd.\\

First, note that $|\prod_1(S)|=1$ is trivially true for $t=1$.  Now, consider $t=3$.  Let $\alpha=\alpha_1+\alpha_2+\alpha_3+m$ so that $\prod_3(S)=\{x^{\alpha-2\alpha_1}y, x^{\alpha-2\alpha_2}y, x^{\alpha-2\alpha_3}y\}$, where $x^{\alpha-2\alpha_i}=x^{\alpha-2\alpha_j}$ if and only if $\alpha_i$ and $\alpha_j$ differ by a multiple of $\frac{n}{2}$.  Since $0 \leq \alpha_1<\alpha_2<\alpha_3<n$, some two of them must differ by strictly less than $\frac{n}{2}$, so $|\prod_3(S)| \geq 2=\frac{3+1}{2}$.  Equality occurs exactly when $\frac{n}{2}\in \{\alpha_2-\alpha_1, \alpha_3-\alpha_1, \alpha_3-\alpha_2\}$.  (Note that this condition puts no constraint on the last $\alpha_i$.)
\\

Henceforth, consider $t \geq 5$.  Define $\alpha$ as in Case $2$ so that $x^{\alpha}y \in \prod_t (S)$ and $x^{\alpha -2(\alpha_{\frac{t-1}{2}+j}-\alpha_i)}y \in \prod_t (S)$ for all $1 \leq i \leq \frac{t-1}{2}$, $1 \leq j \leq \frac{t+1}{2}$.   Recall the list of $t$ differences from Case $2$.  When we double all the differences, we are guaranteed only $\lceil \frac{t}{2} \rceil =\frac{t+1}{2}$ distinct values modulo $n$.  Hence, $|\prod_t(S)|\geq \frac{t+1}{2}$.
\\

Suppose there is equality.  Then $t-1$ of the $t$ listed differences come in pairs $(c, c+\frac{n}{2})$.  Let $c^{\ast}$ be the ``missing'' difference corresponding to the one unpaired difference.  If we add $c^{\ast}$ to the list of $t$ differences, then we know that only these $t+1$ differences (up to multiples of $n$) are attainable via the types of swaps discussed in the previous cases.  We can distinguish the subcases $0<c^{\ast}<\frac{n}{2}$, $c^{\ast}=\frac{n}{2}$, and $\frac{n}{2}<c^{\ast}<n$.  (We know that $c^{\ast}\neq 0$ since $0$ is never missing.)  For the first and third subcases, $t=5$ will require special treatment.
\\

\textbf{Case $4$A.} $0<c^{\ast}<\frac{n}{2}$.
\\

First, consider $t=5$.  Because of the pairings of elements, we have
$$0<\alpha_3-\alpha_2<\frac{n}{2}=\alpha_3-\alpha_1<\alpha_4-\alpha_1<\alpha_5-\alpha_1<n,$$
where there are no missing differences greater than or equal to $\frac{n}{2}$.  From $0<\alpha_3-\alpha_2<\alpha_4-\alpha_2<\alpha_4-\alpha_1$, we can use the fact that the differences come in pairs to conclude that
$$\alpha_3-\alpha_2, \alpha_4-\alpha_2 \in \{(\alpha_4-\alpha_1)-\frac{n}{2}, (\alpha_5-\alpha_1)-\frac{n}{2}, \frac{n}{2}\}=\{\alpha_4-\alpha_3, \alpha_5-\alpha_3, \alpha_3-\alpha_1\}.$$
Since $\alpha_3-\alpha_2<\alpha_4-\alpha_2$, there are three possible ways to assign these values.  First, if $\alpha_3-\alpha_2=\alpha_4-\alpha_3$ and $\alpha_4-\alpha_2=\alpha_5-\alpha_3$, then we immediately get $$\alpha_3-\alpha_2=\alpha_4-\alpha_3=\alpha_5-\alpha_4.$$
Furthermore, $\alpha_5-\alpha_1<(\alpha_5-\alpha_1)+(\alpha_3-\alpha_2)<n+(\alpha_3-\alpha_2)$ implies $$(\alpha_5-\alpha_1)+(\alpha_3-\alpha_2)=n=2(\frac{n}{2})=2(\alpha_3-\alpha_1).$$
Cancelling gives $\alpha_5-\alpha_2=\alpha_3-\alpha_1$ and $\alpha_2-\alpha_1=\alpha_5-\alpha_3=2(\alpha_3-\alpha_2)$.
We can now evaluate $$\alpha_3-\alpha_1=(\alpha_3-\alpha_2)+(\alpha_2-\alpha_1)=3(\alpha_3-\alpha_2)=\frac{n}{2},$$
which yields $\alpha_3-\alpha_2=\frac{n}{6}$.  We can now read off
$$\alpha_1=b, \quad \alpha_2=b+\frac{2n}{6}, \quad \alpha_3=b+\frac{3n}{6}, \quad \alpha_4=b+\frac{4n}{6}, \quad \alpha_5=b+\frac{5n}{6},$$ where $0 \leq b \leq \frac{n}{6}-1$, and this has the desired form.  Second, if $\alpha_3-\alpha_2=\alpha_4-\alpha_3$ and $\alpha_4-\alpha_2=\alpha_3-\alpha_1$, then we immediately get
$$\alpha_2-\alpha_1=\alpha_3-\alpha_2=\alpha_4-\alpha_3.$$
As above, $(\alpha_5-\alpha_1)+(\alpha_3-\alpha_2)=n=2(\alpha_3-\alpha_1)$ leads to $\alpha_5-\alpha_3=\alpha_2-\alpha_1=\alpha_5-\alpha_4$.  But this implies $\alpha_5=\alpha_4$, which is a contradiction, so this possibility cannot occur.  Third, if $\alpha_3-\alpha_2=\alpha_5-\alpha_3$ and $\alpha_4-\alpha_2=\alpha_3-\alpha_1$, then $\alpha_3-\alpha_1=\alpha_4-\alpha_2<\alpha_5-\alpha_2<\alpha_5-\alpha_1$ implies $\alpha_5-\alpha_2=\alpha_4-\alpha_1$, and we get
$$\alpha_2-\alpha_1=\alpha_4-\alpha_3=\alpha_5-\alpha_4.$$
Moreover, $\alpha_3-\alpha_2=\alpha_5-\alpha_3=2(\alpha_2-\alpha_1)$.  We can now evaluate $\alpha_3-\alpha_1=3(\alpha_2-\alpha_1)=\frac{n}{2}$, which yields $\alpha_2-\alpha_1=\frac{n}{6}$.  Finally, we can read off
$$\alpha_1=b, \quad \alpha_2=b+\frac{n}{6}, \quad \alpha_3=b+\frac{3n}{6}, \quad \alpha_4=b+\frac{4n}{6}, \quad \alpha_5=b+\frac{5n}{6},$$ as desired.  This completes the casework and establishes the result for $t=5$.
\\

Henceforth, consider $t \geq 7$.  We have
$$\alpha_{\frac{t-1}{2}+1}-\alpha_{\frac{t-1}{2}}<\dots<\alpha_{\frac{t-1}{2}+1}-\alpha_2<\frac{n}{2}=\alpha_{\frac{t-1}{2}+1}-\alpha_1<\alpha_{\frac{t-1}{2}+2}-\alpha_1<\dots<\alpha_t-\alpha_1,$$
where there are no missing differences greater than or equal to $\frac{n}{2}$.  As such, $\alpha_{\frac{t-1}{2}+1}-\alpha_2<\alpha_{\frac{t-1}{2}+2}-\alpha_2<\alpha_{\frac{t-1}{2}+2}-\alpha_1$ implies that either $\alpha_{\frac{t-1}{2}+2}-\alpha_2=\frac{n}{2}$ or $\alpha_{\frac{t-1}{2}+2}-\alpha_2=(\alpha_t-\alpha_1)-\frac{n}{2}$.  We consider these two possibilities separately.
\\

If $\alpha_{\frac{t-1}{2}+2}-\alpha_2=\frac{n}{2}$, then $\alpha_{\frac{t-1}{2}+2}-\alpha_{\frac{t-1}{2}+1}=\alpha_2-\alpha_1$, and we can write
$$\frac{n}{2}=\alpha_{\frac{t-1}{2}+2}-\alpha_2<\alpha_{\frac{t-1}{2}+2}-\alpha_1<\dots<\alpha_t-\alpha_1,$$ where there are no missing differences in this list.  Then, by the now-familiar argument presented in Case $1$, we get $\alpha_{\frac{t-1}{2}+3}-\alpha_2=\alpha_{\frac{t-1}{2}+2}-\alpha_1$ and $\alpha_{\frac{t-1}{2}+3}-\alpha_{\frac{t-1}{2}+2}=\alpha_2-\alpha_1$.  Continuing this process gives
$$\alpha_2-\alpha_1=\alpha_{\frac{t-1}{2}+2}-\alpha_{\frac{t-1}{2}+1}=\alpha_{\frac{t-1}{2}+3}-\alpha_{\frac{t-1}{2}+2}=\dots=\alpha_t-\alpha_{t-1}.$$
Next, $\frac{n}{2}+(\alpha_t-\alpha_1)<(\alpha_{\frac{t-1}{2}+3}-\alpha_2)+(\alpha_t-\alpha_1)<n+(\alpha_{\frac{t-1}{2}+2}-\alpha_1)$ implies that $(\alpha_{\frac{t-1}{2}+3}-\alpha_2)+(\alpha_t-\alpha_1)=\frac{3n}{2}=\frac{n}{2}+2(\alpha_{\frac{t-1}{2}+1}-\alpha_1)$.  Substituting $\alpha_{\frac{t-1}{2}+3}-\alpha_2=\alpha_{\frac{t-1}{2}+2}-\alpha_1$ and cancelling gives
$$(\alpha_{\frac{t-1}{2}+2}-\alpha_{\frac{t-1}{2}+1})+(\alpha_t-\alpha_{\frac{t-1}{2}+1})=\frac{n}{2}.$$
We can now use the chain of equalities to write
$$(\alpha_2-\alpha_1)+(\frac{t-1}{2})(\alpha_2-\alpha_1)=(\frac{t+1}{2})(\alpha_2-\alpha_1)=\frac{n}{2},$$
which lets us conclude that $\alpha_2-\alpha_1=\frac{n}{t+1}$.  Setting $b=\alpha_1$ and $d=\frac{n}{t+1}$, we can now read off
$$\alpha_{\frac{t-1}{2}+1}=b+\frac{n}{2}, \quad \alpha_{\frac{t-1}{2}+2}=b+d+\frac{n}{2}, \quad \dots, \quad \alpha_t=b+(\frac{t-1}{2})d+\frac{n}{2}.$$
Substituting $\frac{n}{2}=(\frac{t-1}{2}+1)d$ yields
$$\alpha_{\frac{t-1}{2}+1}=b+(\frac{t-1}{2}+1)d, \quad \alpha_{\frac{t-1}{2}+2}=b+(\frac{t-1}{2}+2)d, \quad \dots, \quad \alpha_t=b+(t-1)d.$$
Let $c^{\ast}=\alpha_{\frac{t-1}{2}+1}-\alpha^{\ast}$.  Then, since the differences come in pairs separated by $\frac{n}{2}$, we have
$$\{\alpha_{\frac{t-1}{2}+1}-\alpha^{\ast}, \alpha_{\frac{t-1}{2}+1}-\alpha_2, \alpha_{\frac{t-1}{2}+1}-\alpha_3, \dots, \alpha_{\frac{t-1}{2}+1}-\alpha_{\frac{t-1}{2}}\}=\{d, 2d, \dots, (\frac{t-1}{2})d\}$$ and hence $$\{\alpha^{\ast}, \alpha_2, \alpha_3, \dots, \alpha_{\frac{t-1}{2}}\}=\{b+d, b+2d, \dots, b+(\frac{t-1}{2})d\}.$$  Putting everything together gives
$$\{\alpha^{\ast}, \alpha_1, \alpha_2, \dots, \alpha_t\}=\{b+kd \mid 0 \leq k \leq t\}$$ where without loss of generality we can take $0 \leq b \leq \frac{n}{t+1}-1$, as desired.
\\

If $\alpha_{\frac{t-1}{2}+2}-\alpha_2=(\alpha_t-\alpha_1)-\frac{n}{2}$, then there are no missing elements in
$$\alpha_{\frac{t-1}{2}+1}-\alpha_{\frac{t-1}{2}}<\dots<\alpha_{\frac{t-1}{2}+1}-\alpha_2<\alpha_{\frac{t-1}{2}+2}-\alpha_2<\frac{n}{2}=\alpha_{\frac{t-1}{2}+1}-\alpha_1<\alpha_{\frac{t-1}{2}+2}-\alpha_1<\dots<\alpha_t-\alpha_1.$$
Applying the argument of Case $1$ to the smaller differences gives
$$\alpha_{\frac{t-1}{2}+2}-\alpha_{\frac{t-1}{2}+1}=\alpha_3-\alpha_2=\alpha_4-\alpha_3=\dots=\alpha_{\frac{t-1}{2}}-\alpha_{\frac{t-1}{2}-1}.$$
The relations on the pairs of differences give us
$$(\alpha_{\frac{t-1}{2}+2}-\alpha_1)-(\alpha_{\frac{t-1}{2}+1}-\alpha_{\frac{t-1}{2}})=\frac{n}{2}=\alpha_{\frac{t-1}{2}+1}-\alpha_1,$$ which becomes $\alpha_{\frac{t-1}{2}+2}-\alpha_{\frac{t-1}{2}+1}=\alpha_{\frac{t-1}{2}+1}-\alpha_{\frac{t-1}{2}}$.  Next, the chain of equalities means that the equation
$$(\alpha_{\frac{t-1}{2}+1}-\alpha_{\frac{t-1}{2}})+(\alpha_{\frac{t-1}{2}+2}-\alpha_2)=\frac{n}{2}=\alpha_{\frac{t-1}{2}+1}-\alpha_1,$$
becomes $2(\alpha_3-\alpha_2)=\alpha_2-\alpha_1$.  We can now compute
$$\alpha_{\frac{t-1}{2}+1}-\alpha_1=(\frac{t+1}{2})(\alpha_3-\alpha_2)=\frac{n}{2},$$
which lets us conclude that $\alpha_3-\alpha_2=\frac{n}{t+1}$.  Setting $b=\alpha_1$ and $d=\frac{n}{t+1}$, we can read off
$$\alpha_1=b, \quad \alpha_2=b+2d, \quad \alpha_3=b+3d, \quad \dots, \quad \alpha_{\frac{t-1}{2}+2}=b+(\frac{t-1}{2}+2)d.$$
As above, considering the pairs of differences lets us fill in the remaining values as
$$\alpha_{\frac{t-1}{2}+3}=b+(\frac{t-1}{2}+3)d, \quad \dots, \quad \alpha_t=b+td,$$
and putting everything together gives
$$\{\alpha_1, \alpha_2, \dots, \alpha_t=\{b+kd \mid 0\leq k \leq t\}\setminus \{b+d\}.$$
So both subcases yield the desired result.
\\

\textbf{Case $4$B.} $c^{\ast}=\frac{n}{2}$.
\\

We know exactly where $c^{\ast}$ falls, so we can write the complete list of differences $$0<\alpha_{\frac{t-1}{2}+1}-\alpha_{\frac{t-1}{2}}<\dots <\alpha_{\frac{t-1}{2}+1}-\alpha_1< c^{\ast}=\frac{n}{2}<\alpha_{\frac{t-1}{2}+2}-\alpha_1< \dots <\alpha_t-\alpha_1<n.$$
From $\alpha_{\frac{t-1}{2}+1}-\alpha_2<\alpha_{\frac{t-1}{2}+2}-\alpha_2<\alpha_{\frac{t-1}{2}+2}-\alpha_1$, we conclude that either $\alpha_{\frac{t-1}{2}+2}-\alpha_2=\alpha_{\frac{t-1}{2}+1}-\alpha_1$ or $\alpha_{\frac{t-1}{2}+2}-\alpha_2=\frac{n}{2}$. We consider these two possibilities separately.
\\

If $\alpha_{\frac{t-1}{2}+2}-\alpha_2=\alpha_{\frac{t-1}{2}+1}-\alpha_1$, then $\alpha_{\frac{t-1}{2}+2}-\alpha_{\frac{t-1}{2}+1}=\alpha_2-\alpha_1$, and there are no missing differences in $$0<\alpha_{\frac{t-1}{2}+1}-\alpha_{\frac{t-1}{2}}<\dots <\alpha_{\frac{t-1}{2}+1}-\alpha_2<\alpha_{\frac{t-1}{2}+2}-\alpha_2,$$ and we can use the argument of Case $1$ to get
$$\alpha_{\frac{t-1}{2}+2}-\alpha_{\frac{t-1}{2}+1}=\alpha_2-\alpha_1=\alpha_3-\alpha_2=\dots=\alpha_{\frac{t-1}{2}}-\alpha_{\frac{t-1}{2}-1}.$$
Since there are no missing differences larger than $\frac{n}{2}$, it is easy to identify the pairs of differences, and, in particular,
$$(\alpha_{\frac{t-1}{2}+2}-\alpha_1)-(\alpha_{\frac{t-1}{2}+1}-\alpha_{\frac{t-1}{2}})=\frac{n}{2}=(\alpha_{\frac{t-1}{2}+3}-\alpha_1)-(\alpha_{\frac{t-1}{2}+1}-\alpha_{\frac{t-1}{2}-1}).$$
After cancellations, this becomes
$$\alpha_{\frac{t-1}{2}}-\alpha_{\frac{t-1}{2}-1}=\alpha_{\frac{t-1}{2}+3}-\alpha_{\frac{t-1}{2}+2},$$
so we can add $\alpha_{\frac{t-1}{2}+3}-\alpha_{\frac{t-1}{2}+2}$ to our chain of equalities.  Thus, we can apply the argument of Case $1$ to
$$\alpha_{\frac{t-1}{2}+2}-\alpha_1=\alpha_{\frac{t-1}{2}+3}-\alpha_2<\alpha_{\frac{t-1}{2}+3}-\alpha_1<\dots<\alpha_t-\alpha_1$$ to get
$$\alpha_2-\alpha_1=\alpha_{\frac{t-1}{2}+4}-\alpha_{\frac{t-1}{2}+3}=\alpha_{\frac{t-1}{2}+5}-\alpha_{\frac{t-1}{2}+4}=\dots=\alpha_t-\alpha_{t-1}.$$

Now, $\alpha_t-\alpha_1<(\alpha_t-\alpha_1)+(\alpha_{\frac{t-1}{2}+1}-\alpha_{\frac{t-1}{2}})<n+(\alpha_{\frac{t-1}{2}+1}-\alpha_{\frac{t-1}{2}})$ implies $(\alpha_t-\alpha_1)+(\alpha_{\frac{t-1}{2}+1}-\alpha_{\frac{t-1}{2}})=n$.  Similarly, $n<(\alpha_t-\alpha_1)+(\alpha_{\frac{t-1}{2}+2}-\alpha_{\frac{t-1}{2}})<n+(\alpha_{\frac{t-1}{2}+2}-\alpha_{\frac{t-1}{2}})$ implies $(\alpha_t-\alpha_1)+(\alpha_{\frac{t-1}{2}+2}-\alpha_{\frac{t-1}{2}})=n+(\alpha_{\frac{t-1}{2}+1}-\alpha_{\frac{t-1}{2}})$.  Substituting for $n$ gives
$$(\alpha_t-\alpha_1)+(\alpha_{\frac{t-1}{2}+2}-\alpha_{\frac{t-1}{2}})=((\alpha_t-\alpha_1)+(\alpha_{\frac{t-1}{2}+1}-\alpha_{\frac{t-1}{2}}))+(\alpha_{\frac{t-1}{2}+1}-\alpha_{\frac{t-1}{2}}),$$
and cancelling yields
$$\alpha_{\frac{t-1}{2}+2}-\alpha_{\frac{t-1}{2}+1}=\alpha_{\frac{t-1}{2}+1}-\alpha_{\frac{t-1}{2}}.$$
So, in fact, 
$$\alpha_2-\alpha_1=\alpha_3-\alpha_2=\dots=\alpha_t-\alpha_{t-1}.$$
We can now evaluate
$$(\alpha_t-\alpha_1)+(\alpha_{\frac{t-1}{2}+1}-\alpha_{\frac{t-1}{2}})=t(\alpha_2-\alpha_1)=n,$$
which implies $\alpha_2-\alpha_1=\frac{n}{t}$.  Setting $b=\alpha_1$ and $d=\frac{n}{t}$, we can read off
$$\alpha_1=b, \quad \alpha_2=b+d, \quad \dots, \quad \alpha_t=b+(t-1)d.$$
But then $$(\alpha_{\frac{t-1}{2}+1}-\alpha_1)-(\alpha_{\frac{t-1}{2}+1}-\alpha_{\frac{t-1}{2}})=\alpha_{\frac{t-1}{2}}-\alpha_1=(\frac{t-1}{2}-1)d=\frac{(t-3)n}{2t}<\frac{n}{2}$$
yields a contradiction because we know that the left-most quantity equals $\frac{n}{2}$.  Thus, this possibility cannot occur.
\\

If $\alpha_{\frac{t-1}{2}+2}-\alpha_2=\frac{n}{2}$, then we have the complete list of differences
$$0<\alpha_{\frac{t-1}{2}+1}-\alpha_{\frac{t-1}{2}}<\dots <\alpha_{\frac{t-1}{2}+1}-\alpha_1<\frac{n}{2}=\alpha_{\frac{t-1}{2}+2}-\alpha_2<\alpha_{\frac{t-1}{2}+2}-\alpha_1< \dots <\alpha_t-\alpha_1<n.$$
Applying the argument from Case $1$ to the larger differences gives
$$\alpha_2-\alpha_1=\alpha_{\frac{t-1}{2}+3}-\alpha_{\frac{t-1}{2}+2}=\alpha_{\frac{t-1}{2}+4}-\alpha_{\frac{t-1}{2}+3}=\dots=\alpha_t-\alpha_{t-1}.$$
So consecutive differences greater than or equal to $\frac{n}{2}$ differ by exactly $\alpha_2-\alpha_1$.  Because of the correspondence of the pairs of differences, we know that consecutive differences less than $\frac{n}{2}$ (including $0$) also differ by $\alpha_2-\alpha_1$.  Hence,
$$\alpha_2-\alpha_1=\alpha_{\frac{t-1}{2}+1}-\alpha_{\frac{t-1}{2}}=\alpha_{\frac{t-1}{2}}-\alpha_{\frac{t-1}{2}-1}=\dots=\alpha_3-\alpha_2$$ can be added to the chain of equalities.  Now, $\alpha_t-\alpha_1<(\alpha_{\frac{t-1}{2}+1}-\alpha_{\frac{t-1}{2}})+(\alpha_t-\alpha_1)<n+(\alpha_{\frac{t-1}{2}+1}-\alpha_{\frac{t-1}{2}})$ implies $(\alpha_{\frac{t-1}{2}+1}-\alpha_{\frac{t-1}{2}})+(\alpha_t-\alpha_1)=n$.  Substituting for $\frac{n}{2}$ leads to:
\begin{align*}
(\alpha_{\frac{t-1}{2}+1}-\alpha_{\frac{t-1}{2}})+(\alpha_t-\alpha_1)&=2(\alpha_{\frac{t-1}{2}+2}-\alpha_2)\\
(\alpha_{\frac{t-1}{2}+2}-\alpha_{\frac{t-1}{2}+1})+(t-1)(\alpha_2-\alpha_1)&=2(\alpha_{\frac{t-1}{2}+2}-\alpha_{\frac{t-1}{2}+1})+2(\frac{t-1}{2}-1)(\alpha_2-\alpha_1)\\
2(\alpha_2-\alpha_1)&=\alpha_{\frac{t-1}{2}+2}-\alpha_{\frac{t-1}{2}+1}
\end{align*}
We can now evaluate
$$\alpha_{\frac{t-1}{2}+2}-\alpha_2=(\alpha_{\frac{t-1}{2}+2}-\alpha_{\frac{t-1}{2}+1})+(\frac{t-1}{2}-1)(\alpha_2-\alpha_1)=(\frac{t+1}{2})(\alpha_2-\alpha_1)=\frac{n}{2},$$
which implies $\alpha_2-\alpha_1=\frac{n}{t+1}$ (where obviously $t+1$ divides $n$).  Setting $b=\alpha_1$ and $d=\frac{n}{t+1}$, we can read off
$$\alpha_1=b, \quad \alpha_2=b+d, \quad \dots, \quad \alpha_{\frac{t-1}{2}+1}=b+(\frac{t-1}{2})d, \quad \alpha_{\frac{t-1}{2}+2}=b+(\frac{t-1}{2}+2)d,$$
$$\alpha_{\frac{t-1}{2}+3}=b+(\frac{t-1}{2}+3)d, \quad \dots, \quad \alpha_t=b+td,$$
where $0 \leq b \leq \frac{n}{t+1}-1$, as desired.
\\

\textbf{Case $4$C.} $\frac{n}{2}<c^{\ast}<n$.
\\

Since there are no ``missing'' elements in $$0<\alpha_{\frac{t-1}{2}+1}-\alpha_{\frac{t-1}{2}}<\dots <\alpha_{\frac{t-1}{2}+1}-\alpha_1<\alpha_{\frac{t-1}{2}+2}-\alpha_1=\frac{n}{2},$$ we can apply the argument from Case $1$ to get $$\alpha_{\frac{t-1}{2}+2}-\alpha_{\frac{t-1}{2}+1}=\alpha_2-\alpha_1=\alpha_3-\alpha_2=\dots=\alpha_{\frac{t-1}{2}}-\alpha_{\frac{t-1}{2}-1}.$$  
\vspace{0 mm}

Next, I claim that $\alpha_{\frac{t-1}{2}+1}-\alpha_{\frac{t-1}{2}}$ can be added to this chain of equalities.  To verify this claim, we must treat $t=5$ and $t \geq 7$ separately.  First, consider $t=5$.  We have $$0<\alpha_3-\alpha_2<\alpha_3-\alpha_1=\alpha_4-\alpha_2<\alpha_5-\alpha_2<\alpha_5-\alpha_1<n,$$ where the fact that the differences come in pairs tells us that $$\alpha_5-\alpha_2, \alpha_5-\alpha_1 \in \{\frac{n}{2}, \frac{n}{2}+(\alpha_3-\alpha_2), \frac{n}{2}+(\alpha_3-\alpha_1)\}.$$
Since $\alpha_5-\alpha_2<\alpha_5-\alpha_1$, there are in fact only three possibilities for assigning these values.  First, if $\alpha_5-\alpha_2=\frac{n}{2}$ and $\alpha_5-\alpha_1=\frac{n}{2}+(\alpha_3-\alpha_2)$, then
$$\alpha_2-\alpha_1=(\alpha_5-\alpha_1)-(\alpha_5-\alpha_2)=\alpha_3-\alpha_2,$$ as desired.  Second, if $\alpha_5-\alpha_2=\frac{n}{2}$ and $\alpha_5-\alpha_1=\frac{n}{2}+(\alpha_3-\alpha_1)$, then $$\alpha_2-\alpha_1=(\alpha_5-\alpha_1)-(\alpha_5-\alpha_2)=\alpha_3-\alpha_1$$ implies $\alpha_2=\alpha_3$, which is a contradiction, so this possibility cannot occur.  Third, if $\alpha_5-\alpha_2=\frac{n}{2}+(\alpha_3-\alpha_2)$ and $\alpha_5-\alpha_1=\frac{n}{2}+(\alpha_3-\alpha_1)$, then we have the complete list of differences
$$0<\alpha_3-\alpha_2<\alpha_4-\alpha_2<\alpha_4-\alpha_1=\frac{n}{2}<\alpha_5-\alpha_2<\alpha_5-\alpha_1<n.$$  Note in particular that $\alpha_3-\alpha_2=(\alpha_5-\alpha_2)-\frac{n}{2}$.  Recall that the differences arise from swappping elements in the product $(x^{\alpha_3}y)(x^{\alpha_1}y)(x^{\alpha_4}y)(x^{\alpha_2}y)(x^{\alpha_5}y)=x^{\alpha}y$, where we are allowed to swap $\alpha_3$, $\alpha_4$, and $\alpha_5$ with $\alpha_1$ and $\alpha_2$.  If we instead begin with the product $(x^{\alpha_3}y)(x^{\alpha_5}y)(x^{\alpha_2}y)(x^{\alpha_4}y)(x^{\alpha_1}y)=x^{\alpha'}y$, then we can use the same argument as before, except that now we are allowed to swap $\alpha_4$ and $\alpha_5$ with $\alpha_1$, $\alpha_2$, and $\alpha_3$.  In particular, we can produce the list of differences
$$0<\alpha_4-\alpha_3<\alpha_4-\alpha_2<\alpha_4-\alpha_1=\frac{n}{2}<\alpha_5-\alpha_2<\alpha_5-\alpha_1<n.$$  Since we still have $|\prod_5(S)|=3$, we can conclude that these differences come in pairs separated by $\frac{n}{2}$, too.  In particular, we have $$\alpha_4-\alpha_3=(\alpha_5-\alpha_2)-\frac{n}{2}=\alpha_3-\alpha_2,$$ as desired.  This concludes the casework and establishes the claim for $t=5$.
\\

Now, we establish the claim for $t \geq 7$.  Since $\alpha_{\frac{t-1}{2}+1}-\alpha_1=\alpha_{\frac{t-1}{2}+2}-\alpha_2$,
$$\alpha_{\frac{t-1}{2}+1}-\alpha_1<(\alpha_{\frac{t-1}{2}+2}-\alpha_2)+(\alpha_{\frac{t-1}{2}+1}-\alpha_{\frac{t-1}{2}})<\frac{n}{2}+(\alpha_{\frac{t-1}{2}+1}-\alpha_{\frac{t-1}{2}}).$$
The two swaps in the middle expression are disjoint (since $n \geq 7$), so we can conclude that $(\alpha_{\frac{t-1}{2}+2}-\alpha_2)+(\alpha_{\frac{t-1}{2}+1}-\alpha_{\frac{t-1}{2}})=\frac{n}{2}=\alpha_{\frac{t-1}{2}+2}-\alpha_1$.  After cancellations, this becomes $\alpha_{\frac{t-1}{2}+1}-\alpha_{\frac{t-1}{2}}=\alpha_2-\alpha_1$, which establishes the claim for $n \geq 7$.
\\

From here (for all $n \geq 5$), putting everything together gives $$\alpha_2-\alpha_1=\dots=\alpha_{\frac{t-1}{2}+2}-\alpha_{\frac{t-1}{2}+1}.$$
We can now evaluate $$\alpha_{\frac{t-1}{2}+2}-\alpha_1=(\frac{t-1}{2}+1)(\alpha_2-\alpha_1)=(\frac{t+1}{2})(\alpha_2-\alpha_1)=\frac{n}{2},$$ which implies that $\alpha_2-\alpha_1=\frac{n}{t+1}$ (where $t+1$ divides $n$).  Setting $b=\alpha_1$ and $d=\frac{n}{t+1}$, we can read off
$$\alpha_1=b, \quad \alpha_2=b+d, \quad \dots, \quad \alpha_{\frac{t-1}{2}+2}=b+(\frac{t-1}{2}+1)d.$$
Let $c^{\ast}=\alpha^{\ast}-\alpha_1$.  Since the differences strictly between $0$ and $\frac{n}{2}$ are exactly $\{d, 2d, \dots, (\frac{t-1}{2})d\}$, the conditions on the pairs of differences give
$$\{\alpha^{\ast}, \alpha_{\frac{t-1}{2}+3}, \alpha_{\frac{t-1}{2}+4}, \dots, \alpha_t\}=\{b+d+\frac{n}{2}, b+2d+\frac{n}{2}, \dots, b+(\frac{t-1}{2})d+\frac{n}{2}\}.$$
Substituting $\frac{n}{2}=(\frac{t-1}{2}+1)d$ yields $$\{\alpha^{\ast}, \alpha_{\frac{t-1}{2}+3}, \alpha_{\frac{t-1}{2}+4}, \dots, \alpha_t\}=\{b+(\frac{t-1}{2}+2)d, b+(\frac{t-1}{2}+3)d, \dots, b+td\}.$$
Finally, putting everything together gives $$\{\alpha^{\ast}, \alpha_1, \alpha_2, \dots, \alpha_t\}=\{b+kd \mid 0 \leq k \leq t\},$$ where without loss of generality we can take $0\leq b \leq \frac{n}{t+1}-1$.
\\

This completes the casework for the subcases.  The sufficiency of the condition $$\{\alpha_1, \dots, \alpha_t\}=\{b+\frac{kn}{t+1} \mid 0\leq k \leq t\}\setminus \{b+\frac{\ell n}{t+1}\}$$ across the board follows as in Case $3$.
\end{proof}
\vspace{1mm}

\textbf{Acknowledgments.} This research was conducted at the University of Minnesota, Duluth REU and was supported by NSF/DMS grant 1650947 and NSA grant H98230-18-1-0010.  The author wishes to thank Joe Gallian for suggesting this problem and providing a supportive work environment.  The author also wishes to thank Andrew Kwon, Aaron Berger, and Joe Gallian for reading drafts of this paper and providing valuable advice.

\end{document}